\documentclass[a4paper,12pt]{article}

\usepackage{mathtools}
\usepackage[latin1]{inputenc}
\usepackage[all]{xy}
\usepackage{float}
\usepackage{amsmath,amssymb,amsfonts,amsthm,mathrsfs}
\usepackage{enumerate}
\usepackage{color,fancybox,graphicx}
\usepackage{url}
\usepackage{psfrag}
\usepackage{graphicx,psfrag}
\usepackage{epsfig}
\usepackage{epstopdf}
\usepackage{algorithm, algorithmic}
\usepackage{pstricks}
\usepackage{cancel,ulem}
\renewcommand{\em}{\it} % Pour contrer un renommage de ulem. A faire disrparaitre plus tard
\numberwithin{equation}{section}

\newcommand{\overbar}[1]{\mkern 1.5mu\overline{\mkern-1.5mu#1\mkern-1.5mu}\mkern 1.5mu}

\renewcommand{\le}{\leqslant}
\renewcommand{\leq}{\leqslant}
\renewcommand{\ge}{\geqslant}
\renewcommand{\geq}{\geqslant}

\newcommand{\E}{\mathbb{E}}

\newcommand{\N}{\mathbb{N}}

\newcommand{\R}{\mathbb{R}}

\renewcommand{\P}{\mathbb{P}}
\renewcommand{\L}{\mathbb{L}}

\newcommand{\calF}{\mathcal{F}}

\newcommand{\calX}{\mathcal{X}}

\newcommand{\un}{{\mathbf{1}}}
\renewcommand{\d}{\delta}

\newcommand{\V}{\mathbb V}

\newcommand{\argmin}{{\mathrm{arg}\min}}

\newcommand{\eps}{\varepsilon}
\newcommand{\ph}{\varphi}

\renewcommand{\d}{ {\, \rm d}}

\newcommand{\eqdef}{:=}
\newcommand{\one}{{\mathbf{1}}}

\newcommand{\set}[1]{\left\{#1\right\}}

\newcommand{\p}[1]{ \left(#1\right) }
\renewcommand{\b}[1]{ \left [#1\right ] }

\newcommand{\norm}[1]{\left\Vert#1\right\Vert}

\theoremstyle{plain}
\newtheorem{The}{Theorem}[section]
\newtheorem{Lem}[The]{Lemma}
\newtheorem{Pro}[The]{Proposition}

\newtheorem{Cor}[The]{Corollary}

\newtheorem{Def}[The]{Definition}
\newtheorem{Alg}[The]{Algorithm}

\newtheorem{Ass}{Assumption}
\newtheorem{AssN}{Assumption}

\numberwithin{equation}{section}

\theoremstyle{definition}
\newtheorem{Rem}[The]{Remark}

\begin{document}

\begin{center}
{\sc \Large On the Asymptotic Normality of Adaptive Multilevel Splitting\footnote{This work was partially supported by the French Agence Nationale de la Recherche, under grant ANR-14-CE23-0012, and by the European Research Council under the European Union's Seventh Framework Programme (FP/2007-2013) / ERC Grant Agreement number 614492.}}
\vspace{0.5cm}

\end{center}

{\bf Fr\'ed\'eric C\'erou\footnote{Corresponding author.}}\\
{\it INRIA Rennes \& IRMAR, France }\\
\textsf{frederic.cerou@inria.fr}
\bigskip

{\bf Bernard Delyon}\\
{\it Universit\'e Rennes 1 \& IRMAR, France }\\
\textsf{bernard.delyon@univ-rennes1.fr}
\bigskip

{\bf Arnaud Guyader}\\
{\it Sorbonne Universit\'e \& CERMICS, France }\\
\textsf{arnaud.guyader@upmc.fr}
\bigskip

{\bf Mathias Rousset}\\
{\it INRIA Rennes \& CERMICS, France }\\
\textsf{mathias.rousset@inria.fr}
\bigskip

\medskip

\begin{abstract}
\noindent {\rm
Adaptive Multilevel Splitting (AMS for short) is a generic Monte Carlo method for Markov processes that simulates rare events and estimates associated probabilities. Despite its practical efficiency, there are almost no theoretical results on the convergence of this algorithm. The purpose of this paper is to prove both consistency and asymptotic normality results in a general setting. This is done by associating to the original Markov process a level-indexed process, also called a stochastic wave, and by showing that AMS can then be seen as a Fleming-Viot type particle system. This being done, we can finally apply general results on Fleming-Viot particle systems that we have recently obtained.
%The Adaptive Multilevel Splitting, in the last particle version, is a rare event particle simulation method working as follows. $N$ particles (or replicas) are simulated according to a stopped Markov dynamics. The particle with minimal score, defined by the maximum of a level function over its trajectory is killed. Then a randomly chosen other particle is split in two and one offspring is re-sampled, starting from the entrance time of the current score. The method can estimate the probability of reaching a given score, as well as the distribution of the process conditioned by the latter event. In this paper, we prove a Central Limit Theorem on such estimators in the large sample size limit, using a recent similar CLT result we have obtained in~\cite{cdgr2} for Fleming-Viot type particle systems.
\medskip

\medskip

}

\end{abstract}

\tableofcontents
%----------------------------------------------------------------------
\section{Introduction}\label{intro}
In this article, we prove asymptotic results for the Adaptive Multilevel Splitting (AMS) algorithm used to estimate rares events or to simulate conditionally on rare events. This method belongs to the family of importance splitting algorithms, a set of techniques that date back to Kahn and Harris~\cite{kahn} and Rosenbluth and Rosenbluth~\cite{dim21} to analyze particle transmission energies and molecular polymer conformations. The adaptive version of this method was proposed in~\cite{cg2}. Here we consider the last particle version of this algorithm, introduced in~\cite{ghm} and presented in~\cite{cglp} in the context of molecular dynamics. Recently, this algorithm has been successfully applied to real world chemical computations in~\cite{tmsl16} as well as to Monte Carlo particle transport problems~\cite{ldlrd}.\medskip

To our knowledge, there are almost no theoretical results on the 
convergence of this algorithm, with the notable exception of the idealized case \cite{MR3446035,bgt14,MR3417480}. 
Note however that estimators of unnormalized averages are known to be unbiased in wide generality (see~\cite{bgglr15}). 
Under regularity assumptions discussed below, we give in the present paper an $L^2$-estimate as well as a Central Limit Theorem (CLT). In both cases, we consider the real algorithm, and not the idealized case. We also discuss the asymptotic variance given by the CLT.\medskip

The general framework is as follows. Given a stopped Markov process $(Y_s)_{s\geq 0}$ in a space $E$ and a function $\xi:E\to\R$ such that $\xi(Y_0)=0$ almost surely, the goal is to compute the probability that $\sup_s\xi(Y_s) > 1$ (the rare event), and the distribution of $Y$ given that $\sup_s\xi(Y_s) > 1$. In this context, AMS is an interacting particle system consisting of $N$ particles/trajectories $\p{ Y^{n} }_{n=1 \ldots N}$ simulated according to the distribution of the underlying process $Y$. At each iteration, the particle with minimal score with respect to $\xi$ is killed and another particle is cloned, so that the number of particles/trajectories remains constant and equal to $N$. The algorithm is stopped as soon as all particles have reached the level set $\{\xi > 1\}$. Then the probability is estimated through the number of iterations, and the final empirical distribution estimates the law of $Y$ conditioned by the event $\sup_{s} \xi(Y_s) > 1$.\medskip

The CLT that we obtain for this algorithm applies in the large population limit, that is, when $N$ goes to infinity. This CLT heavily relies on a CLT for Fleming-Viot particle systems that we have recently obtained~\cite{cdgr2}. The key point here is to remark that the AMS algorithm can be recast as a Fleming-Viot particle system by introducing a \textit{level-indexed process}, also called a stochastic wave in~\cite{MR0682731}, associated to the pair $(Y,\xi)$. The latter is obtained through a discontinuous \textit{time change}, where the levels induced by $\xi$ play the role of a new time parameter, and the associated particle state is given by the first entrance in successive level sets. \medskip

The CLT is obtained for diffusions under three main assumptions (referred to as Assumptions~\ref{ass:feller}, \ref{ass:up} and ~\ref{ass:minor}) on the pair $(Y,\xi)$. These assumptions include the case where $Y$ is a diffusion in $\R^d$ satisfying a stochastic differential equation (SDE) of the form
\begin{equation}
\label{eq.diffusion.V}
\d Y_s= b(Y_s) \d s + \sigma(Y_s) \d W_s,
\end{equation}
with smooth coefficients $(b,\sigma)$, and $\xi$ is a smooth function with compact level sets satisfying everywhere some non-degeneracy condition of the form $(\nabla \xi)^T  \sigma \neq 0$. \medskip

In particular, as explained in~\cite{tmsl16}, this algorithm can be applied to simulate so-called reactive trajectories in real-world chemical applications. To fix ideas, consider an overdamped system solution to the SDE 
$$\d Y_s= - \nabla V(Y_s) \d s + \sqrt{2 \beta^{-1}} \d W_s,$$
where $V$ is the interaction energy of the system, and $\beta^{-1}$ is the  temperature. Then, let $A\subset\{\xi<0\}$ denote a ``metastable" state, that is a thin energy level set around a local minimum of $V$. In this context, $\xi$ is called a ``reaction coordinate" and parametrizes a chemical reaction starting from an initial configuration modeled by $A$ up to a final configuration defined by $\set{ \xi > 1}$ (see Figure \ref{algo0xi}). Typically, the system undergoes a large number of quick excursions between the disjoint sets $A$ and $\set{\xi = 0}$. The latter may be simulated on the one hand, defining the initial distribution $\eta_0$ of $Y_0$ on $\set{\xi = 0}$ as an associated stationary distribution. Then, one needs to simulate the reactive trajectories defined as ${\cal L}\set{(Y_s)_{s \geq 0} | S_1 < S_A}$, which represents the rare event of interest. In particular, the associated mean time $\E\b{S_1 | S_1 < S_A}$ and  the probability $\P(S_1< S_A)$ are crucial for the estimation of the underlying chemical kinetics. It turns out that AMS is particularly efficient to estimate such quantities. The interested reader can find details and simulations in~\cite{tmsl16}.\medskip 

\begin{figure}
\begin{center}
\input{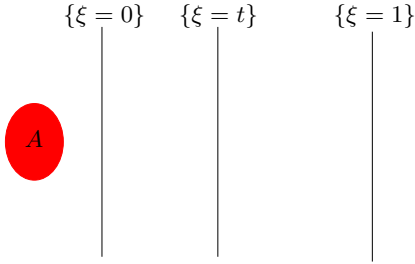}
\caption{The metastable state $A$ and the reaction coordinate $\xi(y_1,y_2)=y_1$.}
\label{algo0xi}
\end{center}
\end{figure}

Reformulating the AMS algorithm as a Fleming-Viot particle system and then applying the CLT for Fleming-Viot particle systems of~\cite{cdgr2} is in fact a quite generic method that may be applied to other types of underlying processes $Y$, for instance to diffusions with degenerate condition $(\nabla \xi)^T  \sigma = 0$, or to Piecewise Deterministic Markov Processes. However, in these cases, defining the associated level-indexed process and checking the assumptions of the CLT for the Fleming-Viot particle system requires extensive, specific analysis that is left for future work. \medskip

The paper is organized as follows. In Section~\ref{sec:set_result}, we introduce the AMS algorithm and state the main result of the article, namely Theorem \ref{gamma}. As mentioned before, the assumptions for the latter are illustrated on the diffusive case. In Section~\ref{sec:FV}, we reformulate the AMS algorithm as a Fleming-Viot particle system built from a so-called level-indexed process. Finally, applying~\cite{cdgr2}, this allows us to establish the desired result. Most of the proofs and technical results are gathered in the appendices.
%We will also discuss the discrete time case, i.e.~the case of Markov chains.

\section{Setting, algorithm, and main result}\label{sec:set_result}

\subsection{Setting}

Let $E$ denote a Polish state space. If $\xi: E \to \R$ is a measurable function and $I$ a subset of $\R$, we denote
$$ \set{\xi \in I} = \xi^{-1}(I)=\set{y \in E,\ \xi(y) \in I}.$$

Besides, if $\mu$ is a probability distribution, $Y$ a random variable with law $\mu$, and $\ph:E\to\R$ a test function, we write
$$\V_\mu(\ph)\eqdef  \V(\ph(Y))=\E[\ph(Y)^2]-\E[\ph(Y)]^2=\mu(\ph^2)-\mu(\ph)^2.$$

Let $(Y_s)_{s \geq 0}$ denote a time homogeneous Markov process with continuous trajectories in $E$ that may be defined from any initial condition $y_0 \in E$. We also assume that the mapping $\xi: E \to \R$, called a level function, is continuous. In what follows, we suppose for simplicity that the law $\eta_0:={\cal L}(Y_0)$ is supported by the level set $\{\xi=0\}$, meaning that 
\begin{equation}\label{eq:initial}
\eta_0 (\xi=0)=1.
\end{equation}

%Time homogeneity is not restrictive, since time may be included in the state space by setting as usual $Y_s=(s,\tilde{Y}_s)$, where $\widetilde{Y}$ is possibly time-inhomogeneous.\medskip 

For each $t\geq 0$, we denote the first entrance time 
in levels strictly greater than $t$ by
\[
  S_t \eqdef \inf \left\{s \geq 0,\  \xi(Y_s) > t \right\} \in [0,+\infty].
\]
Note that by continuity of $\xi$ and $Y$, for all $t\geq 0$ with $S_t < + \infty$, we have
\begin{equation}\label{aeicvazich}
\xi(Y_{S_{t}}) = t.
\end{equation}

Let $A$ denote a Borel set in $E$. By convention, in all what follows, the process $(Y_s)_{s \geq 0}$ is stopped at the random time $S_1\wedge S_A$ where
\[
  S_A \eqdef \inf \left\{s \geq 0,\  Y_s \in A \right\} \in [0,+\infty].
\]
% that is
% $$ Y_s = Y_{S_1\wedge S_A} \qquad \forall s \geq S_1\wedge S_A .$$
% 

Assuming that $$p_1 \eqdef \P({S_1} < {S_A})>0,$$  the goal of the algorithm is to estimate $p_1$, as well as the conditional distribution ${\cal L}( Y_{S_1}|{S_1} < {S_A})$.\medskip

Specific algorithms have been developed in order to efficiently simulate such events, especially when they are rare. The upcoming section recalls the last particle version of Adaptive Multilevel Splitting algorithm as introduced in~\cite{cglp}. The main goal of this paper is to prove the consistency and the asymptotic normality of this algorithm.\medskip

For simplicity, we will assume that almost surely
\begin{equation}\label{eq:stop_finite}
  S_1\wedge S_A < + \infty,
\end{equation}
which implies that the particles trajectories defined in the AMS algorithm are all defined on finite time intervals. While removing or modifying condition~\eqref{eq:initial} and especially condition~\eqref{aeicvazich} requires substantial changes in the definition of the level-indexed process in Section~\ref{sec:FV}, the condition~\eqref{eq:stop_finite} is merely technical and can be simply removed up to dealing with infinite length trajectories (see Section~\ref{sec:non_stop_path}).

\subsection{Adaptive Multilevel Splitting}\label{sec.algo}

%In practice, the latter condition may be obtained by sampling $Y_{S_0}$ conditional on $S_0 > S_A$ where the distribution of $Y_0$ is supported by $\{\xi \leq 0\}$. \medskip
From now on, the integer $N$ denotes the number of trajectories, also called particles. This sample size will stay unchanged all along the algorithm. Besides, random variables denoted with the superscript $n,j$, for instance $Z^{n,j}$, means that it concerns trajectory with index $n$ at iteration $j \geq 0$. Figure \ref{algo1} illustrates the first two steps of the algorithm in the case where $N=3$.

\begin{Alg}[Adaptive Multilevel Splitting]\label{alg:ams}
We start with a sample of the initial condition of the process $Y$, which means that
\begin{equation}\label{eq:ams_init}Y_0^{1,0},\dots,Y_0^{N,0}\ \overset{\rm i.i.d.}{\sim}\ \eta_0.\end{equation}
%The i.i.d.~assumption on the initial condition may be relaxed to obtain the final $L^2$ estimate and CLT, see Section~\ref{sec:main_result}. \medskip

From each initial condition $Y_0^{n,0}$, we simulate a trajectory $(Y_s^{n,0})_{s\geq 0}$. We recall that the latter is stopped when hitting $A$ or level set $\set{ \xi > 1}$. Set $\tau_0=0$ and then iterate on $j \geq 1$:
\begin{enumerate}[(i)]
  \item For $1\leq n\leq N$, compute the score of each particle, meaning the supremum of the level $\xi$ along each particle's trajectory:
  $$
  \sup_{0 \leq s \leq S_A^{n,j-1} \wedge S_1^{n,j-1}} \, \xi(Y^{n,j-1}_s).
  $$
  Find the particle with the smallest score:
  \begin{equation}\label{eq:min_score}
 \left\{\begin{array}{l}
  N_j\eqdef\argmin_{n=1, \ldots, N} \sup_{0 \leq s \leq S_A^{n,j-1} \wedge S_1^{n,j-1}} \, \xi(Y^{n,j-1}_s)\\
  \tau_j:=\sup_{0 \leq s \leq S_A^{N_j,j-1} \wedge S_1^{N_j,j-1}} \xi(Y^{N_j,j-1}_s)
 \end{array} \right.
  \end{equation}
Under Assumptions~\ref{ass:feller} and~\ref{ass:up} below, a unique particle satisfies~\eqref{eq:min_score} (see Proposition~\ref{pro:well_posed}).
  \item Stop the algorithm if $\tau_j=1$. 
  \item for $n\neq N_j$, set $(Y_s^{n,j})_{s\geq 0}= (Y_s^{n,j-1})_{s\geq 0}$.
  \item Pick an index $M_j$ uniformly at random in $\set{1, \ldots, N} \setminus \set{N_j}$. Replace the trajectory with index $N_j$ with a resampled version of 
  the trajectory with index $M_j$, starting from the hitting time of level $\tau_j$, that is 
  \begin{itemize}
    \item set $\sigma_j\eqdef \inf\{s \geq 0,\ \xi(Y^{M_j,j}_s)>\tau_j\} < + \infty$,  
    \item for $s < \sigma_j$, set $Y^{N_j,j}_s= Y^{M_j,j}_s$,  
    \item for $s \geq \sigma_j$, simulate a new piece of trajectory $(Y_s^{N_j,j})_{s \geq \sigma_j}$ according to the law of the underlying process $(Y_s)_{s \geq 0}$ with initial condition $Y^{M_j,j}_{\sigma_j}$.
  \end{itemize}
\end{enumerate}
\end{Alg}

\begin{figure}
\begin{center}
\input{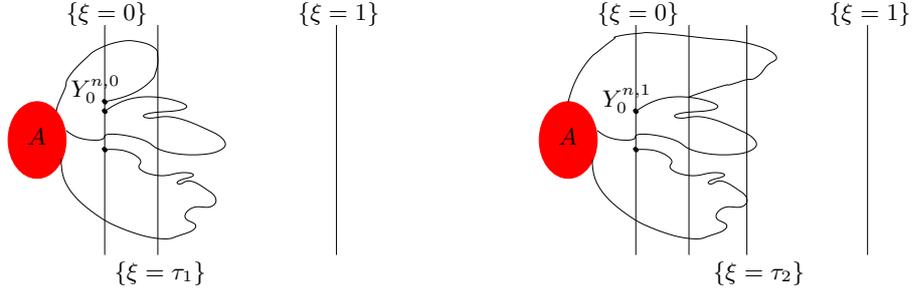}
\caption{The first two steps of AMS with $N=3$ trajectories.}
\label{algo1}
\end{center}
\end{figure}

Assumption~\ref{ass:up} below will ensure that almost surely, in the last step above,  
$$ \forall h > 0, \quad \sup_{s \in [\sigma_j,\sigma_j+h]} \xi(Y^{N_j,j}_s)>\xi(Y^{N_j,j}_{\sigma_j}).$$
In particular, this implies that the sequence $(\tau_j)_{j \geq 0}$ is strictly increasing. Moreover, Assumption~\ref{ass:minor} below will imply that this algorithm stops after a finite number of iterations almost surely (see Proposition~\ref{pro:well_posed}). \medskip

For any $t\in[0,1]$, let us denote $J_t$ the number of branchings of this algorithm between level $0$ and level $t$, that is
$$J_t \eqdef \sup\left\{j,\ \tau_j \leq t\right\},$$
which by definition satisfies
$$ \tau_{J_t} \leq t < \tau_{1+J_t} .$$

\begin{figure}
\begin{center}
\input{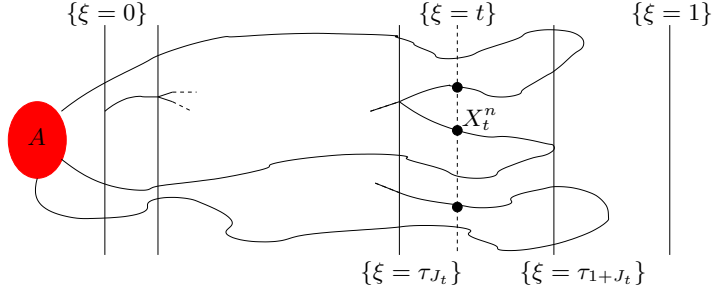}
\caption{The entrance points $X_t^n$ of the level set $\{\xi=t\}$.}
\label{algo2}
\end{center}
\end{figure}

Accordingly, the value of $j$ at the end of the algorithm is $J_1$, and the final particle system is given by the $N$ trajectories $(Y^{n,J_1}_{s})_{s \geq 0}$, $1\leq n\leq N$. By construction, all these trajectories reach the level set $\{\xi=1\}$ and are stopped at this specific time.\medskip
%Let consider level $t \in [0,1]$, and the iteration $J_t$ of the algorithm, defined as the first iteration where all particles have a score $>t$. 

Similarly, for a given level $t \in [0,1]$, the particle trajectories after $J_t$ iterations are given by $(Y^{n,J_t}_{s})_{s \geq 0}$, and the associated entrance times are 
$$
S^n_t \eqdef \inf \{s \geq 0,\  \xi(Y_{s}^{n,J_t}) > t\},
$$
with entrance states $Y^{n,J_t}_{S^n_t}$. To lighten the notation and to prepare the definition of the level-indexed process in Section~\ref{sec:FV}, we denote the latter states (see Figure \ref{algo2})
$$X_t^n:=Y^{n,J_t}_{S^n_t}.$$
Thus, by \eqref{aeicvazich}, one has $\xi(X_t^n)=t$.
Then, for any test function $\ph$, we estimate the law of $Y_{S_t}$ given that $Y_{S_t} < Y_{S_A}$ by the empirical distribution  
$$\eta_t^N(\ph):=\frac{1}{N}\sum_{n=1}^N\ph(Y^{n,J_t}_{S^n_t})=\frac{1}{N}\sum_{n=1}^N\ph(X_t^n).$$
In the same vein, since exactly one trajectory is resampled at each step of the algorithm, our estimator for $p_t=\P(Y_{ S_t}< Y_{S_A})$ is
\begin{align*}p_t^N:=\left(1-\frac{1}{N}\right)^{J_t}.\end{align*}
It was already established in~\cite{bgglr15} (see also~\cite{cdgr2} and the discussion in Section~\ref{sec:FV} of the present article) that $p_t^N \times \eta_t^N(\ph)$ is in fact an unbiased estimator:
$$\E\left[p_t^N \times \eta_t^N(\ph)\right] = \E\left[\ph(Y_{S_t}) \un_{S_t < S_A}\right].$$

\subsection{Assumptions}

In this section, we gather some sufficient conditions to ensure the well-posedness of the previous algorithm and to obtain the main results of Section \ref{sec:main_result}. We illustrate these assumptions in the case of a strong solution of a Stochastic Differential Equation with smooth coefficients. \medskip

%\subsubsection{The assumption set}

Let us begin with some topological and regularity conditions.

\begin{Ass}[Feller regularity]\label{ass:feller}
 $E$ is a locally compact state space, $\xi$ is continuous, $\eta_0 (\xi=0)=1$, $A \subset \set{\xi<0}$, 
 and $(Y_s)_{s \geq 0}$ is a Feller diffusion process, i.e., 
 a Feller process with continuous trajectories.
\end{Ass}

\begin{Rem}
 We recall that Feller processes are strong Markov with respect to their natural filtration denoted $\p{\calF^Y_s = \sigma\p{Y_{s'},0\leq s' \leq s}}_{s \geq 0}$, the latter being 
 necessarily right-continuous  (see for example Theorem 2.7 page 169 in  \cite{ethierkurtz}). %Practical examples of Feller diffusion processes are given by diffusions in $E=\R^d$ with smooth coefficients.
\end{Rem}

For the next assumption, we recall the notation $S_t \eqdef  \inf \{s \geq 0,\   \xi(Y_s) > t \}$ as well as $
S_{B} \eqdef \inf \{s \geq 0,\   Y_s \in B\}$ for any set $B \subset E$. Besides, ${\mathring A}$ and ${\bar A}$ denote respectively the interior and the closure of the set $A$. 

%Note that since $Y$ is Feller, all the latter random times are stopping times with respect to the right-continuous filtration $\calF^Y$.

% We also denote
% $$
% F \eqdef \set{0 \leq \xi\leq 1}.
% $$

\begin{Ass}[Almost sure strict entrance]\label{ass:up}
For any $t \in [0,1]$ and $y \in E$ such that $\xi(y) = t$,
\begin{equation}\label{eq:up}
 \P_{y}\p{S_t = 0} = 1.
\end{equation}
In the same way, for all $y \in \set{ \xi =0}$,
\begin{equation}\label{eq:down}
 \P_{y}\p{ S_{\bar A} = S_{\mathring A} } = 1.
\end{equation}
%Finally, we also assume that $\eta_0$ is such that $\P_{\eta_0}\p{ \widetilde{S}_0 = S_0 } = 1$.
\end{Ass}

By the strong Markov property, (\ref{eq:up}) ensures that $S_t$, defined as the first entrance time in levels strictly greater than $t$, is in fact equal to the hitting time of level $t$, that is
$$S_t = \inf \left\{s \geq 0,\  \xi(Y_s) = t \right\}.$$ 

Besides, since the process $Y$ has continuous trajectories, (\ref{eq:down}) obviously implies that for all $y \in \set{ 0 \leq \xi \leq 1}$,
\begin{equation}\label{eq:downbis}
 \P_{y}\p{ S_{\bar A} = S_{\mathring A} } = 1.
\end{equation}

Moreover, we will show in Lemma~\ref{lem:jump_dens} that (\ref{eq:down}) and the strong Markov property imply that the jump times of the c\`adl\`ag process $t \mapsto Y_{S_t}$ have  atomless distributions. This property is indeed required in~\cite{cdgr2} in order to get the CLT for Fleming-Viot particle systems.\medskip

Let us now define the integral operator
\begin{equation}\label{eq:cont_dir}
 q(\ph)(y) \eqdef \E_y \b{\ph(Y_{S_1}) \un_{S_1 < S_A}}, \qquad y \in \set{0 \leq \xi \leq 1}.
\end{equation}
 Denoting $C_b(\{\xi=1\})$ the set of continuous and bounded functions on the level set $\{\xi=1\}$, we will prove in Lemma~\ref{lem:H1ii} that under Assumptions~\ref{ass:feller} and~\ref{ass:up}, if $\ph \in C_b(\{\xi=1\})$, then $q(\ph)$ is bounded and continuous on $\set{0 \leq \xi \leq 1}$. The proof is based on a general result given in the appendix, namely Lemma~\ref{lem:time_fin}. The integral operator $q$ will prove crucial in the remainder as it will appear in the asymptotic variance of the CLT.\medskip

Our next assumption ensures a uniform control on the probabilities of success, namely $\P_{y}\p{ S_{1} < S_A}$, with respect to the initial condition.

\begin{Ass}[Uniform positive probability of reaching the last level]\label{ass:minor}
We assume that almost surely
$$ S_A \wedge S_1 < + \infty,$$
as well as 
$$  \inf_{y \in \set{\xi = 0} }\P_{y}\p{ S_{1} < S_A} > 0.$$ 
\end{Ass}
First remark that under Assumption~\ref{ass:feller}, by the strong Markov property, one has 
$$  \inf_{y \in \set{0\leq \xi \leq 1} }\P_{y}\p{ S_{1} < S_A} = \inf_{y \in \set{\xi = 0} }\P_{y}\p{ S_{1} < S_A}. $$

As mentioned before, the condition $S_A \wedge S_1 < + \infty$ is a technical simplification of minor significance that may in fact be removed, see Section~\ref{sec:non_stop_path}. \medskip

In Section~\ref{sec:var_ass}, a stronger but easier to check variant of the infimum condition in Assumption~\ref{ass:minor} is presented. \medskip

%\subsubsection{Checking the assumptions for solutions of SDE's}

Let us illustrate the previous conditions in a more specific framework. It turns out that Assumptions~\ref{ass:feller}, \ref{ass:up} and~\ref{ass:minor} are satisfied for elliptic diffusions in a bounded domain. More precisely, let $Y_s \in \R^d$ be a solution to the SDE
\begin{equation}
\label{eq.diffusion.X}
\d Y_s = b(Y_s) \d s+\sigma(Y_s) \d W_s,
\end{equation}
where $b$ and $\sigma$ are functions from $\R^d$ to respectively $\R^d$ and $\R^{d \times n}$, with $n \geq 1$. 
We denote as usual $a = \sigma \sigma^T$. Then we have the following result, whose proof is detailed in Appendix~\ref{sec:proof_diff}. 

\begin{Lem}\label{lem:proof_diff}
Let $Y$ be solution to~\eqref{eq.diffusion.X}. Assume that
\begin{itemize}
 \item[(a)] $\sigma, b$ are in $C^2(\R^d)$ with bounded derivatives of order $i=0,1,2$ on $\set{-1 \leq \xi \leq 1}$.
 \item[(b)] $\xi$ is in $C^2(\R^d)$ with bounded derivatives of order $i=1,2$ on $\set{-1 \leq \xi \leq 1}$. We also assume that $$A ={\bar A}= \set{\xi \leq -1}.$$
 \item[(c)] There exists $\delta > 0$ such that $(\nabla \xi)^T  a \nabla \xi \geq \delta$ on $\set{-1 \leq \xi \leq 1}$.
 %\item[(d)] $\set{\xi = t_0}$ is bounded for some $t_0 \in (-1,0]$. 
\end{itemize}
Then Assumptions~\ref{ass:feller},~\ref{ass:up} and~\ref{ass:minor} hold true.
\end{Lem}

\begin{Rem} Condition~$(c)$ ensures that the martingale part of the process $t \mapsto \xi(Y_t)$ has a strictly positive quadratic variation. It may happen that $(Y_s)_{s \geq 0}$ is Feller, $\xi$ is smooth, but~\eqref{eq:up} does not hold without the addition of Condition~$(c)$. Consider for example the case where $(Y_s)_{s \geq 0}$ is solution to an Ordinary Differential Equation. As a consequence, Assumption~\ref{ass:up} has to be modified without Condition~$(c)$,
% and requires some form of hypo-ellipticity of the diffusion $(Y_s)_{s \geq 0}$. For instance $\xi \in C^1(\R^d)$ with $(\nabla\xi)^T \,a\, \nabla \xi \neq 0$ for all 
% $0 \leq \xi(y) \leq 1$ where $a=a(y)$ is the diffusion matrix of $(Y_s)_{s \geq 0}$ is sufficient (See Lemma~\ref{lem:diffus} below).
% In the case where $(Y_s)_{s \geq 0}$ is an elliptic diffusion, Assumption~\ref{ass:up} can be checked using the strong Markov property and the stronger fact that~\eqref{eq:up} holds true for each $t \in [0,1]$ and $y \in E$ such that $\xi(y) = t$. 
%Note that an extension of~\eqref{eq:up} may hold true even if $(\nabla\xi)^T a \nabla\xi =0 $; 
for instance one may need to resort to an {\it ad hoc} restriction of the state space for which~\eqref{eq:up} is still satisfied. %\note{Cette remarque prepare Langevin}. %See Section~\ref{sec:ext}.
\end{Rem}

\subsection{Main result}\label{sec:main_result}

For any test function $\ph$ and any $t \in [0,1]$, let us define the unnormalized measure $\gamma_t$ by
$$\gamma_t(\ph):=\E\b{\ph(Y_{S_t})\one_{S_t < S_A}},$$
so that $\gamma_0=\eta_0$. Accordingly, the probability that the process $(Y_s)_{s \geq 0}$ reaches level $t$ 
is $p_t:=\gamma_t(\one)=\P(S_t < S_A)$, and the law of $Y_{S_t}$ given that $S_t < S_A$ is denoted $\eta_t$ 
and satisfies $\eta_t(\ph):=\gamma_t(\ph)/\gamma_t(\one)$. \medskip

The purpose of the AMS Algorithm~\ref{alg:ams} is to approximate the previous quantities. 
Namely, for any $t\in[0,1]$, let us recall that the probability $p_t$ is estimated by
$$p_t^N:=\left(1-\frac{1}{N}\right)^{J_t},$$
where $J_t$ denotes the number of iterations necessary to reach level $t$. The measures $\eta_t$ and $\gamma_t$ are respectively estimated by
$$\eta_t^N(\ph):=\frac{1}{N}\sum_{n=1}^N\ph(Y^{n,J_t}_{S^n_t})=\frac{1}{N}\sum_{n=1}^N\ph(X_t^n)\hspace{1cm}\mbox{and}\hspace{1cm}\gamma_t^N(\ph):=p_t^N\eta_t^N(\ph).$$ 
%{\red Ce nouveau paragraphe remplace les demonstrations et r\'ef\'erences diverses \`a Section~\ref{sec:FV}}. We expose our results in several steps. 
%They follow from the discussion in Section~\ref{sec:FV} where Theorem~\ref{th:cdgr2}
%states a result of \cite{cdgr2} concerning the process $X$. In that section, Lemma~\ref{lem:check_ass} 
%checks that $X$ is indeed a Fleming-Viot process and Lemma~\ref{lem:x=y} checks that that assumptions made here on $Y$
%imply the assumption required on $X$ in Theorem~\ref{th:cdgr2}.\medskip

Our first statement is a well-posedness result. As will be explained in Section \ref{oliach}, it is connected to the first point of Theorem \ref{th:cdgr2} and to Lemma \ref{lem:check_ass}.

\begin{Pro}\label{pro:well_posed}
 Under Assumptions~\ref{ass:feller} and~\ref{ass:up}, the AMS Algorithm~\ref{alg:ams} is well-posed in the sense that there is only one particle with minimal score in~\eqref{eq:min_score}. Besides, under Assumption~\ref{ass:minor}, the AMS Algorithm~\ref{alg:ams} is non-explosive in the sense that the algorithm stops after a finite number of iterations almost surely.
\end{Pro}

The second statement is a consistency result in the $L^2$ sense and coincides with the second point of Theorem \ref{th:cdgr2}. 

\begin{Pro} \label{gamma0} Under Assumptions~\ref{ass:feller}, \ref{ass:up}, and~\ref{ass:minor}, for any $\ph \in C_b(\{\xi=1\})$, one has
$$\E\b{ \p{\gamma^N_1(\ph) - \gamma_1(\ph)}^2     } \leq \frac{6 \norm{\ph}_\infty^2}{N}.$$
\end{Pro}

%\note{La prop.3.9 de hardbbb p.15 donne 6  A VERIFIER}
Let us come now to the central limit result, which corresponds to the last point of Theorem \ref{th:cdgr2}. The asymptotic variance is described through the integral operator~\eqref{eq:cont_dir}, namely $q(\ph)(y):=\E_{y}\b{\ph(Y_{S_1})\one_{S_1 < S_A}}$ defined for any $y \in \set{0 \leq \xi \leq 1}$.

\begin{The}\label{gamma} 
Under Assumptions~\ref{ass:feller}, \ref{ass:up}, and~\ref{ass:minor}, for any $\ph \in C_b(\{\xi=1\})$, one has
$$\sqrt{N}\left(\gamma_1^N(\ph)-\gamma_1(\ph)\right)\xrightarrow[N\to\infty]{\cal D}{\cal N}(0,\sigma_1^2(\ph)),$$
where
$$\sigma_1^2(\ph) = p^2_1 \V_{\eta_1}(\ph) - p_1^2\log(p_1) \, \eta_1(\ph)^2 - 2\int_0^1 \V_{\eta_{t}}(q(\ph)) p_t dp_t .$$
\end{The}

Then it is easy to see that Slutsky's lemma and the decomposition
\begin{equation}\label{eq:decomp}
\eta_T^N\p{\ph}-\eta_T(\ph)=\frac{1}{\gamma_T^N(\mathbf{1})} \left(\gamma_T^N(\ph-\eta_T\left(\ph\right))-\gamma_T(\ph-\eta_T\left(\ph\right))\right)
\end{equation}
lead to the upcoming result.

\begin{Cor}\label{eta}
Under Assumptions~\ref{ass:feller}, \ref{ass:up}, and~\ref{ass:minor}, for any $\ph \in C_b(\{\xi=1\})$, one has
$$\sqrt{N}\left(\eta_1^N(\ph)-\eta_1\p{\ph}\right)\xrightarrow[N\to\infty]{\cal D}{\cal N}(0,\sigma_1^2(\ph-\eta_1(\ph))/p_1^2).$$
Besides,
$$\sqrt{N}\left(p_1^N-p_1\right)\xrightarrow[N\to\infty]{\cal D}{\cal N}(0,\sigma^2),$$
where
$$\sigma^2=\sigma_1^2(\mathbf{1}) = -p_1^2\log(p_1) - 2\int_0^1 \V_{\eta_{t}}(q(\mathbf{1})) p_t dp_t.$$
\end{Cor}

First we can remark that all these asymptotic variances can be viewed as the limit of the asymptotic variances for the algorithm with a finite number of levels $0<t_1<\ldots<t_K=1$ as in \cite{alea06,delmoral04a}, when the number $K$ of levels tends to infinity. Details and explanations are provided in \cite{cdgr1}, Section 2.4.\medskip

In the rest of this section, we propose to discuss some consequences of the previous results. We begin with the number of steps of the algorithm. It essentially says that this number grows logarithmically with the rarity of the event and linearly with the number of particles. Remember that one step requires the simulation of only one new trajectory, the computation of its score and the comparison with the $(N-1)$ other already evaluated scores. Hence we can conclude that the total complexity of the algorithm scales like ${\cal O}_P(-N\log N\log p_1)$. A similar remark was already present in \cite{ghm}, but was restricted to the so-called idealized setting.

\begin{Cor}\label{pazeichp}
Under Assumptions~\ref{ass:feller}, \ref{ass:up}, and~\ref{ass:minor}, the number of steps of the AMS algorithm satisfies 
$$J_1=-N\log(p_1)+{\cal O}_P(\sqrt{N}).$$
\end{Cor}

\begin{proof}
Indeed, Proposition \ref{gamma0} with $\ph=\mathbf{1}$ gives $p_1^N=\gamma_1^N(\mathbf{1})$ and $p_1=\gamma_1(\mathbf{1})$, so that
$$p_1^N=p_1+{\cal O}_P(1/\sqrt{N}),$$ 
and
$$\log(p_1^N)=\log(p_1)+{\cal O}_P(1/\sqrt{N}).$$
Besides, 
$$\log(p_1^N)=J_1\log(1-1/N)=J_1(-1/N+o(1/N)).$$ 
Therefore, by using both expressions we get
$$
J_1=-(\log(p_1)+{\cal O}_P(1/\sqrt{N}))(1/N+o(1/N))^{-1}=-N\log(p_1)+{\cal O}_P(\sqrt{N}).
$$
\end{proof}
 
Now we can focus our attention on the asymptotic variance of the probability estimate, that is 
\begin{equation}\label{kshcb}
\sigma^2=\sigma_1(\mathbf{1})^2= -p_1^2\log(p_1) - 2\int_0^1 \V_{\eta_{t}}(q(\mathbf{1})) p_t dp_t.
\end{equation}
By choosing the level function $\xi^\star(y)=q(\mathbf{1})(y)=\P_y(S_1<S_A)$, it turns out that $\eta_t$ is supported on the level set $t$ of $\xi^\star$. Besides, for every $y$ on this level set, we have $q(\mathbf{1})(y)=\P_y(S_1<S_A)=1-t$, so that 
$$\V_{\eta_{t}}(q(\mathbf{1}))=\V_{\eta_{t}}(\xi^\star)=0.$$
Hence the integral term vanishes and $\sigma^2$ reduces to
$$\sigma^2=-p_1^2\log(p_1).$$ 
This function $y\mapsto \xi^\star(y)$ is called the committor function in molecular dynamics, where its prominent role is well known (see e.g.~\cite{hummer-04} and~\cite{metzner-schuette-vanden-eijnden-06}). In fact, the knowledge of the committor function typically requires solving a PDE, which in turn is much more involved than the problem of estimating rare events probabilities. However, it is important to notice that it gives the best possible asymptotic variance.\medskip

This phenomenon also arises when considering the idealized case, where we assume that at each branching, we can generate a new trajectory reaching at least the current level, and independent of the other particles' trajectories (see \cite{bgt14,ghm}). Note also that in the one dimensional case, if $\xi$ is strictly increasing, then the levels sets are reduced to one point, $\eta_t$ is a Dirac measure, and the variance is minimal.\medskip 

On the opposite, we can also exhibit the worst value for $\sigma^2$. For that, consider the variance term in the integrand of (\ref{kshcb}), that is $\V_{\eta_t}(q(\mathbf{1}))$. It corresponds to the variance of the random variable $Z=q(\mathbf{1})(X)$, with $X$ drawn according to $\eta_t$. Hence $Z$ is between $0$ and $1$, and its mean value is $p_1/p_t$. Under those constraints, the largest variance is that of a Bernoulli variable with parameter $p_1/p_t$, and is given by $p_1/p_t(1-p_1/p_t)$. In this situation, we have
$$
- 2\int_0^1 \V_{\eta_{t}}(q(\mathbf{1})) p_t dp_t = -2\int_0^1\frac{p_1}{p_t}(1-\frac{p_1}{p_t})\,p_t\,dp_t
=2p_1^2\log(p_1)-2p_1^2+2p_1,
$$
which yields
$$
\sigma^2\leq 2p_1(1-p_1)+p_1^2\log(p_1)\leq 2p_1(1-p_1).
$$
Notice that this upper bound is exactly twice the variance of a naive Monte Carlo method, which simply consists in simulating $N$ i.i.d.~replicates of the original process $Y$ and counting the proportion of trajectories that reach the level set $\{\xi=1\}$ before $A$. Consequently, we see that if we make a very bad choice for $\xi$, things can go pretty bad, even worse than naive Monte Carlo. Nonetheless, this has to be compared to importance sampling, where one is not even guaranteed to have a finite variance (see for example \cite{GW}). We summarize the previous results in the following lemma.

\begin{Cor}\label{zouchb}
Under Assumptions~\ref{ass:feller}, \ref{ass:up}, and~\ref{ass:minor}, the asymptotic variance for the probability estimator satisfies 
$$-p_1^2\log(p_1)\leq\sigma^2\leq 2p_1(1-p_1).$$
\end{Cor}

Let us conclude this section with some comments on the asymptotic variance for $\eta_1^N(\ph)$, namely $\sigma_1^2(\ph-\eta_1(\ph))/p_1^2$. First notice that for all $t\in[0,1]$ 
$$\gamma_t(q(\ph))=\E[q(\ph(Y_{S_t}))\un_{S_t<S_A}]=\E[\E[\ph(Y_{S_1})\un_{S_1<S_A}|Y_{S_t}]\un_{S_t<S_A}],$$
which yields
$$\gamma_t(q(\ph))=\E[\ph(Y_{S_1})\un_{S_1<S_A}]=\gamma_1(\ph).$$
As a consequence, by linearity of both $q$ and $\eta_t$, one has
$$\eta_t(q(\ph-\eta_1(\ph)))=\eta_t(q(\ph))-\eta_t(q(\un))\eta_1(\ph)=(\gamma_t(q(\ph))-\gamma_1(\ph))/p_t=0.$$ 
Hence, denoting $r_t=p_t/p_1$, we are led to the alternative expression
$$
\sigma_1^2(\ph-\eta_1(\ph))/p_1^2=\V_{\eta_1}(\ph)-2\int_0^1 \eta_t((q(\ph-\eta_1(\ph)))^2)\, r_t\,dr_t.
$$
It is readily seen that $|q(\ph-\eta_1(\ph))|\leq \|\ph-\eta_1(\ph)\|_\infty q(\mathbf{1})$, so that 
$$
-2\int_0^1 \eta_t((q(\ph-\eta_1(\ph)))^2)\, r_t\,dr_t\leq -2 \|\ph-\eta_1(\ph)\|_\infty^2 \int_0^1 \eta_t((q(\mathbf{1}))^2)\, r_t\,dr_t.
$$
Taking into account that $\eta_t(q(\mathbf{1}))=1/r_t$, we get
$$
-2\int_0^1 \eta_t((q(\mathbf{1}))^2)\, r_t\,dr_t = -2\int_0^1 \V_{\eta_t}(q(\mathbf{1}))\, r_t\,dr_t -2\log (p_1).
$$
So we have the following bound
\begin{equation}\label{comvar}
\V_{\eta_1}(\ph)\leq \sigma_1^2(\ph-\eta_1(\ph))/p_1^2 \leq \V_{\eta_1}(\ph) +  \|\ph-\eta_1(\ph)\|_\infty^2 \frac{\sigma^2}{p_1^2}-\log p_1,
\end{equation}
with $\sigma^2$ as in Corollary \ref{eta}. The lower bound is the variance we would get with an i.i.d.~sample from $\eta_1$. As noticed in Corollary \ref{zouchb}, at best the second term in the r.h.s.~of (\ref{comvar}) reduces to $-\|\ph-\eta_1(\ph)\|_\infty^2\log( p_1)$.

\begin{Rem}
 In the AMS Algorithm~\ref{alg:ams}, we assume that the initial condition~(\ref{eq:ams_init}) consists in  $N$ i.i.d.~random variables $Y_0^{n,0}$, $1\leq n\leq N$, with common law $\eta_0$. In fact, this assumption can be relaxed to any exchangeable initial condition satisfying a bound of the form 
 $$
 \E \b{ \left(\eta_0^N(q(\ph))-\eta_0(q(\ph))\right)^2 } \leq \frac{c\norm{\ph}_\infty^2}{N}
 $$
for some constant $c > 0$, as well as the following CLT:
 $$\sqrt{N}\left(\eta_0^N(q(\ph))-\eta_0(q(\ph))\right)\xrightarrow[N\to\infty]{\cal D}{\cal N}(0,\V_{\eta_0}(q(\ph))).$$
In that case, all the results of this section still hold true (see Remark~$2.8$ in~\cite{cdgr2}).
\end{Rem}

\subsection{Extension to path observables and entrance times}

This section deals with an extension of the CLT to richer observables. For this purpose, we can consider the following Polish space.
\begin{Def}\label{def:path_topo} Let $C_{\rm stop}(\R_+,E)$ denote the space of continuous paths, with possibly a given terminal time $s$. We will use the notation
$$
y_{[0,s]} \eqdef
\left\{\begin{array}{ll}
  (y_{s'})_{0 \leq s' \leq s}  &\text{\, if \,} s < +\infty,\\
  (y_{s'})_{s' \geq 0}   &\text{\, if \,} s = +\infty.\\
\end{array}\right.
$$ 
We say that a sequence $y^n_{[0,s^n]}$ in $C_{\rm stop}(\R_+,E)$ converges towards $y_{[0,s]}$ if $\lim_n s^n = s \in {\overbar{\R}}_+$ and $\lim_n (y^n_{s' \wedge s^n})_{s' \geq 0} = (y_{s' \wedge s})_{s' \geq 0}$ in $C(\R_+,E)$ endowed with uniform convergence on compacts. This defines a Polish topology on $C_{\rm stop}(\R_+,E)$ that we will always use in the sequel, unless otherwise specified.
\end{Def}

The main message is that, {\it mutatis mutandis}, the central limit result of Theorem \ref{gamma} is still valid in this new context. More precisely, we have the extended following CLT (see Appendix~\ref{sec:path_proof} for the proof).

\begin{The}\label{th:clt_path}Let $\psi: C_{\rm stop} \p{ \R_+,E }\to \R$
denote a given continuous and bounded functional. Set $\calX^n_t \eqdef Y^{n,J_t}_{[0,S^{n}_t]}$
and 
$$\eta^N_t \eqdef \frac{1}{N} \sum_{n=1}^N \delta_{\calX^n_t}.$$
Besides, consider
$$\gamma_t(\psi) \eqdef \E \b{\psi(Y_{[0,S_t]}) \un_{S_t < S_A}},$$ 
and
$$
q(\psi)(y_{[0,s]}) \eqdef \E \b{\psi(Y_{[0,S_1]}) \un_{S_1 < S_A} | Y_{[0,s]} = y_{[0,s]} }.
$$
Denote as before $\eta_{t} \eqdef \gamma_t / \gamma_t(\un)$ and $\gamma^N_t \eqdef p^N_t \eta^N_t$. If Assumptions~\ref{ass:feller}, \ref{ass:up}, and~\ref{ass:minor} are satisfied, then Proposition~\ref{gamma0}, Theorem~\ref{gamma} and Corollary~\ref{eta} hold true when replacing $\ph$ with $\psi$.
\end{The}

\begin{Rem}
In the special case of entrance times, i.e.~observables of the form $\psi(Y_{S_1},S_1)$, the CLT is in fact a direct consequence of Theorem~\ref{gamma}. Indeed, consider the time homogeneous Markov Feller process 
$$
\widetilde{Y}_h \eqdef (Y_h, s_0+h), \qquad h \geq 0,
$$
defined for each initial condition $\widetilde{Y}_0=(y_0,s_0)\in\{0\leq\xi\leq 1\}\times[0,1]$. By construction, it can be easily checked that if $Y$ satisfies Assumptions~\ref{ass:feller}, \ref{ass:up}, and~\ref{ass:minor}, then it is also true for $\widetilde{Y}$, so that this case is included in~Theorem~\ref{gamma}.
\end{Rem}

\section{Level-indexed processes and Fleming-Viot particle systems}\label{sec:FV}
In this section, we introduce a c\`adl\`ag process $X$ based on the couple $(Y,\xi)$ and called the level-indexed process. In a different framework, it was introduced by Dynkin and 
Vanderbei in~\cite{MR0682731} and called a stochastic wave by these authors. They mainly studied it in the case where $Y$ is a diffusion, but apparently without any specific application in mind. In our framework, thanks to a slight modification of this object, we can interpret the AMS Algorithm~\ref{alg:ams} as a Fleming-Viot particle system. The results of~\cite{cdgr2} on Fleming-Viot particle systems can then be applied in order to prove Proposition~\ref{gamma0} and Theorem~\ref{gamma}. 

\subsection{Level-indexed processes}
Let us denote $\partial$ a cemetery point. Recall that $S_A \eqdef \inf\{ s \geq 0,\ Y_s \in A\}$ and $S_t \eqdef \inf\{ s \geq 0,\ \xi(Y_s) > t\}$. 

\begin{Def}[Level-indexed process]\label{def:lev_ind} %Assume that $A \subset \set{ \xi < 0} $ and 
%Denote by
%$$ F \eqdef \set{ 0 \leq \xi \leq 1} \subset E. $$
Let the first condition of Assumption~\ref{ass:up} be satisfied, namely that for each $t \in [0,1]$ and for each $y \in \set{ \xi = t}$, we have $\P_y(S_t =0)=1$. The level-indexed process, or stochastic wave, $(X_h)_{h \geq 0}$ with state space $F \cup \set{\partial}$ where 
$$ F \eqdef \set{0 \leq \xi \leq 1}, $$
and associated with the pair $(Y,\xi)$ and initial condition $Y_0=x$ is defined by its initial condition $X_0=x$, and for any $h \geq 0$,
\[
X_h \eqdef 
  \begin{cases}
    Y_{S_{(\xi(x) + h)\wedge 1}}  &\mbox{if}\  \ S_{(\xi(x) + h)\wedge 1} < S_A,  \\
    \partial &\mbox{if}\  \ S_{(\xi(x) + h)\wedge 1} \geq S_A.
  \end{cases}
\]
\end{Def}

\begin{Rem} \label{rem:lev_ind}
\begin{itemize}
  \item The first condition of Assumption~\ref{ass:up} is necessary to ensure the consistency of the definition of the level-indexed process. For instance it is necessary to ensure that $X_0=Y_{S_{\xi(x)}} = x$ given the initial condition $X_0=Y_0=x$.
  \item If $Y$ has continuous trajectories, $\xi$ is continuous, and $S_1 \wedge S_A < +\infty$ as is the case here, then $X_h$ is c\`adl\`ag and $\xi(X_h) = \p{\xi(x) + h} \wedge 1$ for all $h \geq 0$. See also Section~\ref{sec:non_stop_path} for the case where $S_1 \wedge S_A = +\infty$ with non zero probability.
  \item If the initial level is $\xi(x) = 0$ and if $t \in [0,1]$ is such that $X_t\neq\partial$, then $\xi(X_t)=t$ (see Figure \ref{algo3}). In particular, if $X_1\neq\partial$, then $X_t=Y_{S_t}$ for all $t \in [0,1]$.
  
\begin{figure}
\begin{center}
\input{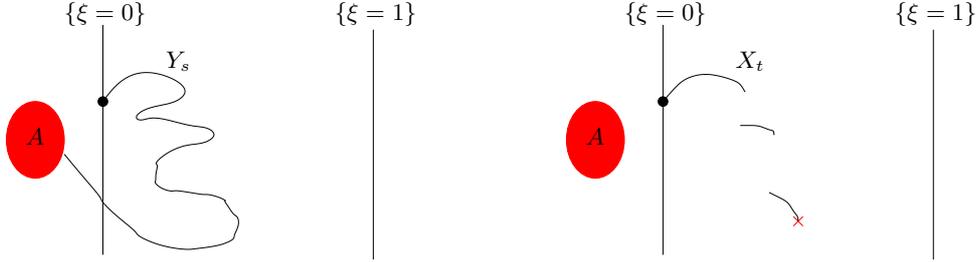}
\caption{The level-indexed process $X_t$ associated to the pair $(Y,\xi)$.}
\label{algo3}
\end{center}
\end{figure}

  \item If $Y$ is Feller, as is the case here, then $Y$ is strong Markov with respect to its right-continuous natural filtration $\calF^Y$. By construction of the level-indexed process, it implies that $X$ is - at least - a time homogeneous Markov process with respect to the filtration $\p{\calF^Y_{S_{(\xi(X_0)+h) \wedge 1}}}_{h \geq 0}$, and thus {\it a fortiori} with respect to its smaller natural filtration. %$\calF^X$. (**** FAIRE UN LEMME ??? ****)
  \item If $F=\set{0 \leq \xi \le1}$ is compact, then the continuity Lemma~\ref{lem:time_fin} implies that the level-indexed process $X$ is itself Feller.
\end{itemize}
\end{Rem}

In the case where $Y$ is not stopped at $S_{A}$, the level-indexed process has been introduced in~\cite{MR0682731} and called a stochastic wave.
If, for example, $Y_s=(Y_s^1,Y_s^2)$  is a two-dimensional Brownian motion with $Y_0 = 0$
and $\xi(y^1,y^2)=y^1$, then $X_t^1=t$ and $X_t^2=Y^2_{S_t}$ where $S_t = \inf \set{s \geq 0, Y^1_s = t}$ is a symmetric Cauchy process
with a dense set of discontinuity points. This is illustrated on Figure~\ref{lip}.\medskip

\begin{figure}
\begin{center}
\includegraphics[height=7.3cm,width=14cm]{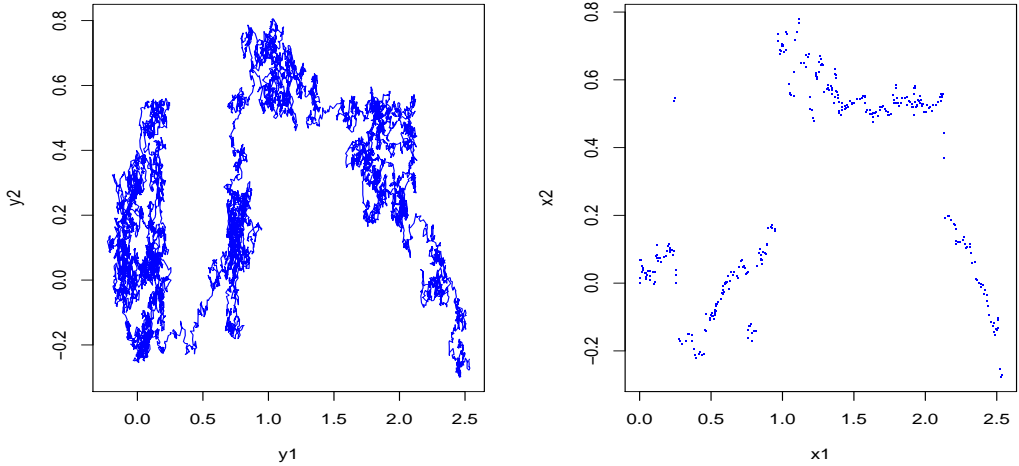}
\caption{2D Brownian trajectory $Y_s=(Y_s^1,Y_s^2)$ and associated stochastic wave $(X_t^1,X_t^2)=(t,Y^2_{S_t})$ when $\xi(y^1,y^2)=y^1$.}
\label{lip}
\end{center}
\end{figure}

%\begin{figure}[h]
%
%\centering
%   \includegraphics[scale=0.2]{bm1}
%   \caption{\label{fig:bm1} 2D Brownian trajectory $Y_t=(Y_t^1,Y_t^2)$.}
%%\end{figure}
%%\begin{figure}[h]
%\centering
%   \includegraphics[scale=0.2]{lip1}
%   \caption{\label{fig:lip1} Level-indexed trajectory (stochastic wave) $(X_t^1,X_t^2)=(t,Y^2_{S_t})$ for the trajectory in Figure~\ref{fig:bm1}.}
%\end{figure}

\begin{Rem}[Soft versus hard killing times]
 It turns out that under Assumption~\ref{ass:up}, the killing time of the level-indexed process is typically ``soft'' in the sense that it is a totally inaccessible stopping time, i.e.~a stopping time that cannot be predicted (see~\cite{js03} for a precise definition, as well as the discussion in~\cite{cdgr2}). This is for instance a consequence of the Feller property when $X$ is Feller. Note that this is a stronger property than having an atomless distribution. 
 %{\red Exemple \`a ajouter par Matthias}
 Interestingly, the CLT in~\cite{cdgr2} also holds true for ``hard'' killing times, so that it may be used to treat cases beyond Assumption~\ref{ass:up}. 
 \end{Rem}

The Markov semi-group of the level-indexed process, defined by

$$
Q^h\ph(x) \eqdef \E[\ph(X_h) | X_0=x],
$$

can be easily related to the integral operator $q(\ph)(y)=\E_{y}\b{\ph(Y_{S_1})\one_{S_1 < S_A}}$ as follows.
\begin{Lem}\label{lem:dirichlet_op}
For any $x\in F$ and any $\ph: F \to \R$ extended to $F \cup \set{\partial}$ with the convention $\ph(\partial) = 0$, one has
 $$
 Q^{1-\xi(x)}\ph(x) = q(\ph)(x).
 %\E[\ph(Y_{S_{1}})\one_{S_{1}<S_A}}|Y_0=x]=\E[\ph(X_1)\one_{\tau_\partial>t}|X_t=x]
 $$
\end{Lem}

%*** PROOF ?? ***
% \begin{Rem}
% If we apply this semi-group to the level-indexed process we get the martingale (using that $\xi(X_t)=t$)
% $$
% Q^{1-t}\ph(X_t)=\E[\ph(Y_{S_{1-t+\xi(X_t)}})\one_{S_{1-t+\xi(X_t)<S_A}}|Y_0=X_t]=\E[\ph(X_1)\one_{\tau_\partial>t}|X_t].
% $$
% \end{Rem}

\subsection{AMS as a Fleming-Viot particle system}

The AMS Algorithm~\ref{alg:ams} can be recast in the form of a Fleming-Viot algorithm as studied in~\cite{cdgr2}. For this purpose, let us consider a time homogeneous c\`adl\`ag Markov process $\p{X_{h}}_{h \geq 0}$ in $F \cup \set{\partial}$, constructible from any initial condition in $F$. We assume that $\partial$ is an absorbing state, meaning that $X_{h'} \in \partial$ whenever $X_{h} \in \partial$ and $h' \geq h$. Let us first recall what we mean by Fleming-Viot particle system.

\begin{Def}[Fleming-Viot particle system]\label{def:FV}
An exchangeable particle system $(X_t^1, \ldots, X^N_t)_{t \geq 0}$ in $F^N$ is called the {\em Fleming-Viot} particle system associated with $\p{X_{h}}_{h \geq 0}$ if:
\begin{itemize}
\item Initialization: the particles are initially i.i.d.~with distribution $\eta_0$ 
$$X_0^1,\dots,X_0^N\ \overset{\rm i.i.d.}{\sim}\ \eta_0,$$
\item Evolution and killing: between branching times, each particle evolves independently according to the law of the underlying Markov process $X$ until one of them hits $\partial$,
\item Branching (or rebirth): the killed particle is taken from $\partial$ and is instantaneously given the state of one of the $(N-1)$ other particles - the choice being uniformly random,
\item And so on until time 1.
\end{itemize}
\end{Def}

Note that in order to be well-defined, a Fleming-Viot particle system should almost surely satisfy the following two properties: (i) particles die at different times, (ii) there is a finite number of branchings in the time interval $[0,1]$. Some conditions ensuring (i) and (ii) are given and discussed below. \medskip

The next result makes explicit the connection between the AMS algorithm, Fleming-Viot particle systems and the level-index process. 

\begin{Lem} \label{lem:x=y}
Let Assumptions~\ref{ass:feller} and~\ref{ass:up} hold true for the pair $(Y,\xi)$. Recall that the particles have initial level $0$, i.e.~$\eta_0 \p{\xi = 0}=1$. Consider the AMS Algorithm~\ref{alg:ams}. For each $n = 1,\ldots,N$ and each $t\in[0,1]$, set as before
 $$
S^{n}_t \eqdef \inf \set{ s \geq 0, \,  \xi(Y^{n,J_t}_s) > t} =\inf \set{ s \geq 0, \,  \xi(Y^{n,J_t}_s) = t},
 $$
 as well as 
 $$
 X^{n}_t \eqdef Y^{n,J_t}_{S^{n}_t}. %\qquad \text{for, \,} \tau_{j-1}\leq t < \tau_j,
 $$
 %Consider the filtration
 %$$
 %\calF^N_t \eqdef \sigma \set{ Y^{n,j}_{S^{n,j}_t} }
 %$$
Then $(X_t^1, \ldots, X^N_t)_{t \geq 0}$ is the Fleming-Viot particle system in $F =\set{0 \leq \xi \leq 1}$ associated with the level-indexed process $X$ of the pair $(Y,\xi)$ in the sense of Definition~\ref{def:lev_ind} and Definition~\ref{def:FV}. 
%Then $(X_t^1, \ldots, X^N_t)_{t \geq 0}$ is the level-indexed process of $(Y,\xi)$ in the sense of Definition~\ref{def:lev_ind} and is a Fleming-Viot process in $F \eqdef \set{0 \leq \xi \leq 1}$ associated to the Markov dynamics of the process $Y_{S_t}$ in the sense of Definition~\ref{def:FV}. }
\end{Lem}

\begin{proof}
For $t \in [0,1]$, $ j\geq 0$ and $n =1,\ldots,N$, let us first define
$$S^{n,j}_t:=\inf\{s\geq 0,\ \xi(Y_s^{n,j})\in A\cup \{\xi=t\}\},$$
and 
$$
 X^{n,j}_t \eqdef  \begin{cases}
 Y^{n,j}_{S^{n,j}_t} &\mbox{if}\ \xi(Y^{n,j}_{S^{n,j}_t})=t\\
 \partial &\mbox{otherwise}
   \end{cases}
 $$
By Assumption~\ref{ass:up} and Remark~\ref{rem:lev_ind}, the initial condition satisfies 
$$(X_0^{1,0},\ldots,X_0^{N,0})=(Y_0^{1,0},\ldots,Y_0^{N,0})  \in \{\xi=0\}^N,$$ so that 
$$\xi(X_0^{1,0})=\dots=\xi(X_0^{N,0})=0.$$
Note also that for all $t \in [0,1]$, if $X^{n,j}_t\neq\partial$, then
$$
\xi(X^{n,j}_t)=\xi(Y^{n,j}_{S^{n,j}_t}) = t,
$$
so that $X^{n,j}$ is indeed the level-indexed process associated with $Y^{n,j}$ in the sense of Definition~\ref{def:lev_ind}. \medskip

Set $\tau_0=0$. By construction of the AMS Algorithm~\ref{alg:ams}, the processes $(X_t^{1,j},\ldots,X_t^{N,j})_{0 \leq t \leq 1}$ can thus be iteratively constructed for $j \geq 1$ as follows:
\begin{enumerate}[(i)]
%  \item For each $1\leq n\leq N$, compute the {\em score} of particle $n$, defined as its killing time 
%    $$
%  \tau_{n,j} \eqdef \inf \{0 \leq t \leq 1,  X^{n,j}_t = \partial \}.
%  $$
%  \item Find the particle with the smallest score:
%  \begin{equation}\label{eq:min_score}
%  N_j\eqdef\argmin_{n=1, \ldots, N} \tau_{n,j}\hspace{1cm}\mbox{and}\hspace{1cm}\tau_j:=\tau_{N_j,j}.
%  \end{equation}
\item We can reformulate $N_j$ and $\tau_j$ defined in the AMS Algorithm~\ref{alg:ams} as
 \begin{equation}
  \left\{\begin{array}{l}
  N_j\eqdef\argmin_{n=1, \ldots, N} \sup_{0 \leq s \leq S_A^{n,j-1} \wedge S_1^{n,j-1}} \, \xi(Y^{n,j-1}_s)\\
  \tau_j:=\sup_{0 \leq s \leq S_A^{N_j,j-1} \wedge S_1^{N_j,j-1}} \xi(Y^{N_j,j-1}_s)
 \end{array} \right.
  \end{equation}
 %and denote $N_j$ the index of the killed particle.
\item Stop if $\tau_j = 1$, i.e.~all trajectories are still alive at time $1$. %Otherwise, we assume that a unique particle $N_j$ is killed at time $\tau_j<1$.
%\item for $n\neq N_j$, set $(X_t^{n})_{0\leq t\leq \tau_j}= (X_t^{n})_{0\leq t\leq \tau_j}$.
\item Set $X^{n,j} \eqdef X^{n,j-1}$ for $n \neq N_j$. Pick a number $M_j$ uniformly at random in $\set{1, \ldots, N} \setminus \set{N_j}$.
\item Replace the trajectory on $[0,\tau_j]$ of the particle with index $N_j$ with the trajectory of the particle with index $M_j$, that is set $(X_t^{N_j,j})_{0\leq t\leq \tau_j} \eqdef (X_t^{M_j,j})_{0\leq t\leq \tau_j}$.
Let particle $N_j$ evolve independently starting from state $X_{\tau_j}^{N_j,j}$ at time $\tau_j$, until time $1$ or until it is killed. % \note{J'ai modifie cette etape: seule la particule $N_j$ est mutee. Ca permet de montrer directment l'identitee avec AMS, et c'est la version la plus simple. Evidemment, muter tout le monde est equivalent.}
\end{enumerate}
If we now set
$$ X^n_t \eqdef X^{n,J_t}_t =  X^{n,j-1}_t \qquad \text{for \, } \tau_{j-1} \leq t < \tau_{j}$$
for $n = 1, \ldots, N$ and $ j\geq 1$, we thus obtain by definition the Fleming-Viot particle system associated with the level-indexed Markov process $X$ of Definition~\ref{def:FV}.
%\end{Def}
\end{proof}

\subsection{$L^2$ estimate and CLT for Fleming-Viot particle systems}\label{oliach}

Building on~\cite{cdgr2}, we can now present two sufficient assumptions to obtain the desired $L^2$-estimate and CLT for Fleming-Viot particle systems based on the level-indexed processes. \medskip

The first assumption is the following:
\renewcommand{\theAssN}{(\~A)}
\begin{AssN}\label{ass:H1} This assumption has two parts.
\begin{enumerate}
\item[(i)] For any initial condition $x \in F=\{0\leq\xi\leq 1\}$, the jump ``times'' of the level-indexed process $(X_h)_{h \geq 0}$ have an atomless distribution:
$$
\P\p{ X_{h^-} \neq X_h \mid X_{0} = x } = 0 \qquad \forall x \in F,\ \forall h \geq 0.
$$
\item[(ii)] If $\ph \in C_b(\{\xi=1\})$, then the mapping $x \mapsto q(\ph)(x)=Q^{1-\xi(x)}\ph(x)$ is continuous on $F$.
\end{enumerate}
\end{AssN}

The second key assumption is simply:
\renewcommand{\theAssN}{(B)}
\begin{AssN}\label{ass:H2} The Fleming-Viot particle system is well-defined in the sense that $\P(J_1 < + \infty)=1$, where $J_1$ denotes the number of branchings until final time $1$.
\end{AssN}

Under these assumptions, \cite{cdgr2} implies the following (see Section~\ref{sec:cdgr2} for details on how to rigorously import the content of~\cite{cdgr2}).

\begin{The}\label{th:cdgr2} Under Assumptions~\ref{ass:H1} and~\ref{ass:H2}, one has the following:
\begin{itemize}
 \item The Fleming-Viot particle system is well-posed in the sense that only one particle is killed at each branching time. %and {\red\cancel{$\P(\calN_1 < + \infty)=1$} $\P(J_1 < + \infty)=1$}. {\red *** Cet item est bizarre, il dit : si on a H2, alors H2 est vérifiée...***}
 \item $L^2$ estimate: for any $\ph \in C_b(\{\xi=1\})$,
 $$\E\b{ \p{\gamma^N_1(\ph) - \gamma_1(\ph)}^2     } \leq \frac{6 \norm{\ph}_\infty^2}{N}.$$
 \item Central Limit Theorem: for any $\ph \in C_b(\{\xi=1\})$, 
$$\sqrt{N}\left(\gamma_1^N(\ph)-\gamma_1(\ph)\right)\xrightarrow[N\to\infty]{\cal D}{\cal N}(0,\sigma_1^2(\ph)),$$
where
$$\sigma_1^2(\ph) = p^2_T \V_{\eta_1}(\ph) - p_1^2\log(p_1) \, \eta_1(\ph)^2 - 2\int_0^1 \V_{\eta_{t}}(Q^{1-t}(\ph)) p_t dp_t.$$
\end{itemize}
\end{The}

Proposition~\ref{gamma0}, Theorem~\ref{gamma} and Corollary~\ref{eta} are then consequences of the following:
\begin{Lem}\label{lem:check_ass}
 Assumptions~\ref{ass:feller} and \ref{ass:up} imply Assumption~\ref{ass:H1}. With the addition of Assumption~\ref{ass:minor}, they also imply Assumption~\ref{ass:H2}.
\end{Lem}

The proof of Lemma~\ref{lem:check_ass} is given in Section~\ref{sec:check_ass}. For now, let us just give some intuition behind this result. First, assume for simplicity that $\xi(X_0) = 0$ and that a jump of the level-indexed process occurs, i.e.~$X_{t^-}\neq X_t$ for a given $t \in [0,1]$. Then,
since by left continuity $X_{t^-}=Y_{S_{t^-}}$, this jump means that $s \mapsto \xi(Y_s)$ has a local maximum with value $t$. However, under Assumption~\ref{ass:up}, this is prohibited since the $Y$-hitting times of levels $>t$ and $\geq t$ are equal almost surely.\medskip
%First, the jumps of the level-indexed process, i.e. $X_{t^-}\neq X_t$ for a given $t$, are related to the fact that the function $s \mapsto \xi(Y_s)$ has a local maximum with value $t$. However, under Assumption~\ref{ass:up}, this is prohibited since the $Y$-hitting times of levels $>t$ and $\geq t$ are equal almost surely.\medskip

Second, the continuity of $x \mapsto q(\ph)(x)=\E_x \b{\ph(Y_{S_1})\one_{S_1 < S_A}}$ is a consequence of the fact that $Y$ is Feller (Assumption~$1$) and that for Feller processes, hitting times of the interior or the closure of, respectively, $A$ and $\set{ \xi \geq 1}$ are the same.\medskip

Finally, the fact that the algorithm has almost surely a finite number of branchings (non-explosion) comes directly from the uniform lower bound of Assumption~\ref{ass:minor} through a comparison with a geometric random sequence.
%Assumption~(J) may be checcked using the following
%Defining the explosion level $\tau_{\rm exp} \eqdef \inf \p{0 \leq t \leq 1 \mid \calN_t < +\infty}$, the Fleming-Viot particle system $(X^1_t, \ldots, X^N_t)_{0 \leq t \leq 1}$ is defined by setting $X^n_t \eqdef \lim_{j \to + \infty} X_t^{n,j}$. \medskip

% {\tiny
% Proposition~\ref{gamma0} and \ref{gamma} are then resulting from direct application of Theorem~{2.7??} and Proposition~{3.9??} in~\cite{cdgr2} to the above particle system. \medskip
% 
% In order to use the result above, we need to extend the definition of the level-indexed process to any intial conditions $x \in F$, and any level $t \geq 1$. For this purpose, for any $X_0=x_0 \in F$ and any $h \geq 0$ ,  we set more generally
% \[
% X_h \eqdef 
%   \begin{cases}
%     Y_{S_{t}}  &\mbox{if}\,  S_{t} < S_A,  \\
%     \partial &\mbox{if}\,  S_{t} \geq S_A,
%   \end{cases}
%   \qquad
%   \text{where} \, t := (h - \xi(x_0)) \wedge 1.
% \]
% which is a time homogeneous Markov process.
% }
%********** METTRE ICI QUE LES DEUX ALGOS SONT LES MEMES P.S. !!! *******

\appendix

\section{Preliminaries on Feller processes}

In this section, we recall some standard properties of continuous Feller processes. Most proofs are detailed in~\cite{cdgr2} in the case of c\`adl\`ag processes.

\begin{Def} Let $E$ be a locally compact Polish space. Let $C_0(E)$ denote the space of continuous functions that 
vanish at infinity. A continuous process $(Y_s)_{s \geq 0}$ in $E$ is Feller if each of its probability transition maps $C_0(E)$ into itself: for all $\ph \in C_0(E)$ and $s \geq 0$, $\p{y \mapsto \E_y[\ph(Y_s)]} \in C_0(E)$.
\end{Def}
Feller processes have many useful standard properties: (i) The associated natural filtration $\calF^Y_s \eqdef \sigma\p{Y_{s'}, \, 0 \leq s' \leq s}$ is right-continuous; (ii) $Y$ is strong Markov with respect to $\calF^Y$. \medskip

%In particular, $S_A$ and $S_t$ are stopping times for each $t \in [0,1]$, and $Y$ is strong Markov with respect to the latter. 
% (iv) $Y$ is quasi-left continuous: if $\sigma_\infty$ is a stopping time that is predictable, i.e.~$\sigma_\infty=\lim_n \sigma_n$ with $\sigma_n < \sigma_\infty$ for $n \geq 0$ being a predicting sequence of stopping times, then $Y$ never jumps at $\sigma_\infty$:
% $$
% \P\p{Y_{\sigma_\infty^-} = Y_{\sigma_\infty}} = 1.
% $$
% Note that deterministic times are predictable. \medskip
We will need the following slightly less standard pathwise continuity of Feller processes.
% that is using the usual Skorokhod $J_1$ topology.
% \begin{Pro}
% Let $d$ be a compatible metric on $E$ and let $C(\R_+,E)$ enotes the space of c\`adl\`ag maps from $\R_+$ to $E$. There is a Polish topology on $C(\R_+,E)$, called the Skorokhod $J_1$ topology, characterized by the following: $\lim_n y^n = y^\infty$ in $C(\R_+,E)$ if and only if there is a sequence $(\lambda^n)_{n \geq 0}$ of increasing one-to-one maps of $\R_+$ such that for each $s_0 \geq 0$
% $$\lim_n \sup_{0 \leq s \leq s_0} d \p{y^n_{\lambda^n(s)},y^\infty_s} =\lim_n \sup_{0 \leq s \leq s_0} \abs{\lambda^n(s)-s} =0.$$
% \end{Pro}
%The distribution of $(Y_s)_{s \geq 0} \in C(\R_+,E)$ is continuous with respect to the distribution of its initial condition $Y_0$. 
%as detailed in the following lemma.
\begin{Lem}\label{lem:cont_feller}
Let $C(\R_+,E)$ denote the space of continuous trajectories endowed with uniform convergence on compacts. 
Let $\p{Y^y_{s}}_{s \geq 0} \in C(\R_+,E)$ denote a given Feller process with initial condition $Y_0=y$. 
Then the mapping $y \mapsto {\cal L}\p{{\p{Y^y_{s}}_{s \geq 0}}}$ from $E$ to probabilities on $C(\R_+,E)$, 
endowed with convergence in distribution, is continuous.
\end{Lem}
\begin{proof} In~\cite{cdgr2} Lemma~$4.3$, the convergence is shown in the Skorokhod space instead of $C(\R_+,E)$ using Theorem~$17.25$ of~\cite{Kall}. The Skorokhod topology and the topology of uniform convergence on compacts on $C(\R_+,E)$ are known to be the same on continuous trajectories,
% Let $(y^n)_{n \geq 0}$ be a sequence of initial conditions with $\lim_n y^n = y$ and denote $Y^n \eqdef Y^{y^n}$ as well as $Y \eqdef Y^{y}$. By Theorem~$17.25$, Condition~(iv) of~\cite{Kall} (Condition~(ii) which is trivially true here, implies Condition~(iv)) the sequence of processes $\p{ \p{Y^{n}_{s}}_{s \geq 0}}_{n \geq 0}$ convergence in distribution towards $\p{Y_{s}}_{s \geq 0}$ in the Skorokhod space $D(\R_+,E)$. 
see Lemma~$10.1$, Chapter~$3$ of~\cite{ethierkurtz} (see also Problem~$7$ Chapter~VI in~\cite{pollard}). % implies that $Y^n$ converges in distribution towards $Y$ also in $C(\R_+,E)$. 
Hence the result.
% It thus remains to check that the sequence of distributions of $\p{ \p{Y^{n}_{s}}_{s \geq 0}}_{n \geq 0}$ is tight in . This is then a classical consequence of the topological identity $C(\R_+,E) = \set{y \in \D(\R_+,E) \mid y \, \text{continuous} }$ where $\D(\R_+,E)$ is the usual Skrokhod space, and Aldous tightness criteria, see Kallenberg, Theorem~$17.25$ or Jacod Shiryaev. 
\end{proof}

We then recall some lower and upper semi-continuity of hitting times with respect to the locally uniform topology
\begin{Lem}\label{lem:stop_time_cont} Let $B \subset E$ be a Borel set.
 For each $y\in C(\R_+,E)$, define $ s_{\mathring{B}}(y) \eqdef \inf \{s \geq 0,\   y_s \in \mathring{B}\},$
 as well as $ s_{\bar{B}}(y)  \eqdef \inf \{s \geq 0,\ y_s \in \bar{B}\}.$ 
 Then $s_{\mathring{B}}$ is upper semi-continuous in $C(\R_+,E)$
 and $s_{\bar{B}}$ is lower semi-continuous in $C(\R_+,E)$:
 for any  sequence  $(y^n)_{n \geq 1}$ converging to $y\in C(\R_+,E)$,
 \begin{align*}& \limsup_n s_{\mathring{B}} \p{y^n} \leq s_{\mathring{B}}\p{ y}, \\
 &s_{\bar{B}} \p{y} \leq \liminf_n s_{\bar{B}} \p{y^n}.
 \end{align*}
\end{Lem}
\begin{proof}
See Lemma~$4.4$ in~\cite{cdgr2}.
% Denote $y^\infty = \lim_n y^n$. \medskip
%  \underline{Upper continuity.} By continuity of $y^\infty$, there is a sequence $(s_m)_{m \geq 0}$ such that $\lim_m {s_m} = s_{\mathring{B}}\p{y^\infty}$ and $y^\infty_{s_m} \in \mathring{B}$. By definition of the uniform topology, up to a subsequence extraction, one has $y^m(s_m) \in \mathring{B}$ and thus $s_{\mathring{B}}\p{y^m} \leq s_m$ and the result follows. \medskip
%  
%  \underline{Lower continuity.} Let $d$ be a compatible metric on $E$. By definition of $s_{\bar{B}}$, for each $\eps > 0$, there is a $\delta > 0$ such that $d(Y_s,B) \geq \delta$ for all $0 \leq s \leq s_{\bar{B}}(y^\infty) - \eps$. By definition of uniform topology, $d(Y^n_s,B) \geq \delta$ for all $0 \leq s \leq s_{\bar{B}}\p{y^\infty} - \eps$ for $n$ large enough. This implies that $s_{\bar{B}}\p{y^n} \geq s_{\bar{B}}\p{y^\infty} - \eps$ for $n$ large enough, hence the result since $\eps$ is arbitrary.
\end{proof}

We can then conclude with the general property used to prove the continuity of the integral operator $q$ defined in~\eqref{eq:cont_dir} (see Lemma~\ref{lem:H1ii}). We denote as before $
S_{\bar B}  \eqdef \inf \{s \geq 0,\  Y_s \in \bar{B}\} $ as well as $S_{\mathring B} \eqdef \inf \{s \geq 0,\   Y_s \in \mathring B\}$.

\begin{Lem}\label{lem:time_fin}
  Let $B \subset E$ be a Borel set, $Y$ be a continuous Feller process, and $\lim_n y^n= y$ a converging sequence of initial conditions. If
  \begin{equation}\label{eq:time_cond}
    \P_{y} \p{S_{\bar B} = S_{\mathring B}} =1,
 \end{equation}
then the distribution of $S_B$ under $\P_{y^n}$ converges when $n \to +\infty$ towards its distribution under $\P_{y}$. \medskip
If moreover $\P_y(S_B < + \infty) > 0$, then the joint distribution of $(S_B,Y_{S_B})$ in $\R_+ \times E$ under $\P_{y^n}\p{\quad | S_B < + \infty}$ converges when $n \to +\infty$ towards the joint distribution under $\P_{y}\p{\quad | S_B < + \infty}$.
\end{Lem}

\begin{proof}
Using Lemma~\ref{lem:cont_feller} and a Skorokhod embedding argument, a sequence $(Y^n_s)_{s \geq 0}$ of Feller processes with initial conditions $(y^n)_{n \geq 0}$ can be constructed on a single probability space so that $\lim_n Y^n = Y$ in $C(\R_+,E)$ almost surely, where $Y$ denotes the Feller process with initial condition $y$. Then Lemma~\ref{lem:stop_time_cont} with~\eqref{eq:time_cond} implies that $\lim_n S^n_{B} = S_B$, hence the first result. The second result follows by continuity of $Y$.
%Finally, the definition of the Skorokhod topology implies that $Y^n_{S^n_O} \in O$ is relatively compact with adherence $\set{ Y_{S_O^-} , Y_{S_O} }$; but~\eqref{eq:time_cond_2} implies that a subsequence cannot converge to $Y_{S_O^-}$ if $Y_{S_O^-} \neq Y_{S_O}$.
\end{proof}

In order to obtain a pathwise version of the main CLT of the present paper, we will need 
a pathwise version of the latter continuity result. For this purpose, let us recall that the Polish space 
  $C_{\rm stop}(\R_+,E)$ of continuous paths with a possibly given end time (see Definition~\ref{def:path_topo}) 
is equipped with a topology defined by the convergence of end times and of processes stopped at the end time uniformly 
on any finite time intervals. \medskip

The following technical lemma about the continuity of the extension of paths will prove useful.
\begin{Lem}\label{lem:map_ext} The extension map 
$$
\begin{array}{rcl}
T\ :\ C_{\rm stop}(\R_+,E) \times C(E,\R_+)  &\longrightarrow & C(E,\R_+) \\
y= \p{ (y_{s'})_{0 \leq s' \leq s} , (\tilde{y}_{h})_{h \geq 0} } & \longmapsto & Ty= \left\{\begin{array}{ll} y_{s'}& s' \leq s \\  \tilde{y}_{s'-s}&  s' \geq s \end{array}\right.
\end{array}
$$
defined for paths satisfying $y_s=\tilde{y}_0$ is continuous.
\end{Lem}
\begin{proof}
Let $d$ stand for the distance on $E$. Denote by $(y^n_{[0,s^n]})$ a sequence of paths in $C_{\rm stop}(\R_+,E)$ converging to $y^\infty_{[0,s^\infty]}$
(for the topology given in Definition~\ref{def:path_topo}), 
and $(\tilde{y}^n)$  a sequence of paths in $C(E,\R_+)$ converging to $\tilde{y}^\infty$
(uniformly on compact sets). We assume that $y^n_{s^n}=\tilde{y}^n_0$ for all $n$, as well as $y^\infty_{s^\infty}=\tilde y^\infty_0$.
We have to prove the convergence of the extended function sequence $(Ty^n)$ to $Ty^\infty$.
Note that for any $s$, one can control $d(Ty^\infty_s,Ty^n_s)$ by considering all cases:
\begin{itemize}
\item If $s\le s^\infty\wedge  s^n$, 
\begin{align*}d(Ty^\infty_s,Ty^n_s)\le d(y^\infty_{s\wedge s^\infty},y^n_{s\wedge s^n}).\end{align*}

\item If $s\ge s^\infty\vee s^n$, 
\begin{align*}d(Ty^\infty_s,Ty^n_s)
\le  d(\tilde y^\infty_{s-s^\infty},\tilde y^n_{s-s^n})
\le  d(\tilde y^\infty_{s-s^\infty},\tilde y^\infty_{s-s^n})+ d(\tilde y^n_{s-s^n},\tilde y^\infty_{s-s^n}).\end{align*}

\item If $s^n\le s\le s^\infty$, 
\begin{align*}
d(Ty^\infty_s,Ty^n_s)&= d( y^\infty_s,\tilde y^n_{s-s^n}) \\
& \le d( y^\infty_{s},\tilde y^\infty_{s -  s^n})+d(\tilde y^n_{s-s^n},\tilde y^\infty_{s-s^n}) \\
& \le d( y^\infty_{s}, y^\infty_{s^\infty}) + d( \tilde y^\infty_{0},\tilde y^\infty_{s -  s^n})+d(\tilde y^n_{s-s^n},\tilde y^\infty_{s-s^n}). 
\end{align*}

\item If $s^\infty\le s\le s^n$, 
\begin{align*}
d(Ty^\infty_s,Ty^n_s)= d(\tilde y^\infty_{s-s^\infty},y^n_s)
\le d( y^\infty_{s\wedge s^\infty},y^n_{s\wedge s^n})+d(\tilde y^\infty_{s-s^\infty},\tilde y^\infty_0).
\end{align*}
\end{itemize}

Let $s_0 \geq 0$ be given. The convergence assumptions, together with the uniform continuity of $Ty^\infty$ on compacts, imply that, when $n$ goes to infinity, all the right hand sides converge uniformly  to $0$ with respect to $s \in [0,s_0]$. Hence the result. 
%+++++++++++++++++++
%Denote by $(y^1_{[0,s^1]},y^2_{[0,s^2]})$ two paths in $C_{\rm stop}(\R_+,E)$, extended on $C(E,\R_+)$ 
%using the pair of paths $(\tilde{y}^1,\tilde{y}^2)$ and the above defined map. Let $s_0 \geq s^1 \vee s^2$ be given. 
%Then
%\begin{align*}
% \sup_{0 \leq s' \leq s_0} d\p{y^1_{s'},y^2_{s'}} \leq \sup_{0 \leq s'\leq s_0} d\p{y^1_{s' \wedge s^1},y^2_{s' \wedge s_2}} + \gamma^1_{s_0}(\abs{s_2-s_1}) + \sup_{0 \leq h \leq s_0} d\p{\tilde{y}^1_{h },\tilde{y}^2_{h}}
%\end{align*}
%where in the above $\gamma^1_{s_0}$ denotes the uniform modulus of continuity of $y^1$ on $[0,s_0]$. 
%As a consequence if both $y^2_{[0,s^2]}$ and $\tilde{y}^2$ converges respectively to $y^1_{[0,s^1]}$ 
%and $\tilde{y}^1$, so is the extension on $C(E,\R_+)$, hence the continuity of the extension map.
\end{proof}

We can then safely prove the following pathwise continuity of stopped Feller processes.

\begin{Lem}\label{lem:time_fin_path} Let $ \lim_n y^n_{[0,s^n]} = y_{[0,s]}$ denote a converging sequence of initial path conditions in $C_{\rm stop}(\R_+,E)$. Let $B \subset E$ denote a Borel set and $Y$ a continuous Feller process. Assume that
$$\P_{y_s} \p{S_{\bar B} = S_{\mathring B}} =1.$$%   as well as
%   \begin{equation}\label{eq:time_cond_2}
%     \P_{y^\infty} \p{Y_{S_O^-}=Y_{S_O^{}} \, \text{or} \, \, Y_{S_O^-} \notin \bar{O} } =1.
%   \end{equation}
  Then the distribution of $(Y_{0 \leq s \leq S_B})$ in $C_{\rm stop}(\R_+,E)$  under 
%  $\P \b{ \,.\, | Y_{[0,s^n]}=y^n_{[0,s^n]}}$ 
  $\P(\quad | Y_{[0,s^n]}=y^n_{[0,s^n]})$ converges when $n \to +\infty$ towards its distribution under $\P(\quad | Y_{[0,s]}=y_{[0,s]})$.
\end{Lem}
\begin{Rem}
Recall that, by Definition~\ref{def:path_topo}, if  $S_B=+\infty$, then $(Y_{0 \leq s \leq S_B})$ is actually $(Y_{0 \leq s <+\infty})$. 
\end{Rem}
\begin{proof} It is an extension of the proof of Lemma~\ref{lem:time_fin} using Lemma~\ref{lem:map_ext}. 
Indeed, the latter and a Skorokhod embedding argument allow us to construct a sequence $Y^n$ converging almost surely to $Y$ 
in $C(E,\R_+)$ such that (i) $Y^n_{s'}=y^n_{s'}$ for $s' \leq s^n$ and $Y_{s'}=y_{s'}$ for $s' \leq s$; 
(ii) all processes are distributed according to $Y$ with initial conditions prescribed by~(i). 
Then Lemma~\ref{lem:stop_time_cont} with \eqref{eq:time_cond} implies that $\lim_n S^n_{B} = S_B$ almost surely, 
hence the result.
\end{proof}

\section{Proof of Assumptions~\ref{ass:feller}, \ref{ass:up}, and \ref{ass:minor} for diffusions in $\R^d$}\label{sec:proof_diff}

We can now establish Lemma~\ref{lem:proof_diff}, by checking successively that Assumptions~\ref{ass:feller}, \ref{ass:up}, and \ref{ass:minor} hold true under the conditions~(a),~(b) and~(c) stated in Lemma~\ref{lem:proof_diff}. \medskip

{\sc Step~1:}  Assumption~\ref{ass:feller} holds true. \medskip

Indeed, Condition~(a) implies that the diffusion is a strong solution of the SDE~\eqref{eq.diffusion.X} and is Feller, see for example \cite{ethierkurtz} Th.2.4 page 373. \medskip

{\sc Step~2:} Assumption~\ref{ass:up} holds true. \medskip

By definition of the stopping times $S_t$, condition \eqref{eq:up}~of Assumption~\ref{ass:up} will follow from
\begin{align}\label{ppp}
\P_y\Big(\varlimsup_{s \downarrow 0}s^{-\frac12}(\xi(Y_s)-\xi(y))=+\infty\Big)=1
\end{align}
for any $y \in \set{0 \leq \xi \leq 1}$. On the other hand, recalling that $A \eqdef \set{\xi \leq -1}$, condition \eqref{eq:down}~of Assumption~\ref{ass:up} follows similarly from the strong Markov property for Feller processes and from the fact that, for any $y \in \set{\xi = -1}$,
\begin{align}\label{qqq}
\P_y\Big(\varliminf_{s \downarrow 0}s^{-\frac12}(\xi(Y_s)-\xi(y))=-\infty\Big)=1.
\end{align}
We claim that both~\eqref{ppp} and~\eqref{qqq} hold true for any $y \in \set{-1 \leq \xi \leq 1}$, which will conclude the proof of Step~$2$. \medskip 

Since $\xi$ is $C^2$, It\^o's formula gives
 \begin{align*}
\d \xi(Y_s) 
&= \Big(\nabla\xi(Y_s)^Tb(Y_s)+\tfrac12Tr(\nabla^2\xi(Y_s)a(Y_s))\Big)\d s
+\nabla\xi(Y_s)^T\sigma(Y_s)\d W_s,
\end{align*}
so that if we denote by $U_s \d s$ the finite variation part of this decomposition, and 
set 
$$\Sigma_s=\sqrt{\big((\nabla \xi)^T  a \nabla \xi\big)(Y_s)}\geq \sqrt{\delta} > 0,$$
then
for the one dimensional Brownian motion $\d \widetilde{W}_s = \Sigma_s^{-1} \nabla\xi(Y_s)^T\sigma(Y_s)\d W_s$, it comes
 \begin{align}\label{eq:diff_lev}
\d \xi(Y_s)&=U_s\d s+\Sigma_s\d \widetilde{W}_s.
\end{align} 
Since $s \mapsto U_s$ is continuous, it remains to prove that
\begin{align}\label{pppp}
\P_y\Big(\varliminf_{s \downarrow 0}s^{-\frac12}\int_0^{s}\Sigma_{r}\d \widetilde{W}_r=-\infty\Big) = \P_y\Big(\varlimsup_{s \downarrow 0}s^{-\frac12}\int_0^{s}\Sigma_r\d \widetilde{W}_r=+\infty\Big)=1.
\end{align}
For this, remark that the process 
\begin{align}\label{eq:time_BM}
s \mapsto  B_{\int_0^s \Sigma_r^2 \d r} \eqdef \int_0^s\Sigma_r \d \widetilde{W}_r
\end{align}
is a time-changed Brownian motion $B$ (Chapter~V, Section~$1$ of~\cite{revuzyor}). The law of the iterated logarithm for the Brownian motion (Chapter~II, Section~$1$ of \cite{revuzyor}) now implies (\ref{pppp}) since almost surely we have
\begin{align*}
\varlimsup_{s\rightarrow 0}s^{-\frac12}\int_0^{s}\Sigma_r\d \widetilde{W}_r
&=\varlimsup_{s\rightarrow 0}\Big(\frac1s\int_0^s\Sigma_r^2 \d r\Big)^{\frac12}
\Big(\int_0^s\Sigma_r^2 \d r\Big)^{-\frac12}B_{\int_0^s\Sigma_r^2 \d r}\\
&=\Sigma_0\varlimsup_{s\rightarrow 0}
\Big(\int_0^s\Sigma_r^2 \d r\Big)^{-\frac12}B_{\int_0^s\Sigma_r^2 \d r}\\
&=+\infty.
\end{align*}
The same reasoning applies for the other limit in (\ref{pppp}).\medskip

{\sc Step~$3$}: Assumption~\ref{ass:minor} holds true.

Consider the differential equation~\eqref{eq:diff_lev} above, and recall that Condition~(a) on the coefficients implies that $U_s$ is bounded, while the positive lower bound in Condition~(c) implies that $\Sigma_s > 0$ is bounded from above and from below. \medskip

We then consider the continuous process $s \mapsto Z_s$ defined by $Z_0 = \xi(Y_0)$ and 
$$
d Z_s := - \lambda_{0} \Sigma^{2}_s \, d s + \Sigma_s \, d \widetilde{W}_s,
$$
where $\lambda_0$ is such that 
$$
U_s \Sigma_s^{-2} \leq \lambda_{0} 
$$
almost surely and for all $s \geq 0$. By construction, (i) the process $s \mapsto Z_s - \xi(Y_s)$ is decreasing and thus negative, and (ii) $s \mapsto Z_s$ is the time-changed Brownian motion~\eqref{eq:time_BM}, but drifted with constant drift $-\lambda_0$, that is 
$$
\widetilde{Z}_{l=\int_0^s \Sigma_r^2} := Z_s,
$$
satisfies $d \widetilde{Z}_l = -  \lambda_{0} d l + d B_l$. We will denote by $S^Z_{ \pm 1}$ the first hitting time of $\pm 1$ by $Z$, and $L^{\tilde Z}_{\pm 1}$ the first hitting time of $ \pm 1$ of $\tilde{Z}$ so that
$$
\int_0^{S^Z_{ \pm 1}} \Sigma_r^2 \, d r = L^{\tilde Z}_{\pm 1}.
$$

Consider also the stopping time $\sigma$ defined by
$$
\int_0^\sigma \Sigma_r^2 \, d r =1 .
$$
Notice that $\sigma\leq 1/\delta$ almost surely. Now, let us first prove that
$$p_0 := \sup_{y \in \set{-1 \leq \xi \leq 1} }\P_y \p{S_1=S_{-1}=+\infty} = 0 .$$
Conditioning and applying the strong Markov property yields
\begin{align*}
\P_y \p{S_1=S_{-1}=+\infty} & = \E_y \b{ \E_y[\un_{S_1=S_{-1}=+\infty} | {\cal F}^Y_{\sigma}] \un_{S_1 \wedge S_{-1} > \sigma} } \\
& = \E_y \b{ \E[\un_{S_1=S_{-1}=+\infty} |Y_{\sigma}] \un_{S_1 \wedge S_{-1} > \sigma} } \\
& \leq p_0 \P_y \p{S_1 \wedge S_{-1} > \sigma }  \leq p_0 \P_y \p{S_1 > \sigma }.
\end{align*}
Since $\xi(Y_s) \geq Z_s $ we have that  $S_1 \leq S^Z_1$ so that 
$$ \P_y(S_1 > \sigma ) \leq \P_{y}(S^Z_1 > \sigma ) = 
\P\left(\int_0^{S_1^Z} \Sigma_r^2 dr>\int_0^\sigma \Sigma_r^2 dr\right)
=\P_{y}(L^{\tilde Z}_1 > 1). $$
Since $\tilde{Z}$ starting from $0$ is stochastically smaller than $\tilde{Z}$ starting from $\xi(y) \geq 0$, it yields $\P_{y}(L^{\tilde Z}_1 > 1) \leq \P(L^{\tilde Z}_1 > 1 | \tilde Z = 0) < 1 $, so that 
$$
p_0 \leq p_0 \times \P(L^{\tilde Z}_1 > 1 | \tilde Z = 0),
$$
which shows that $p_0=0$.\medskip

Finally, let us  prove that 
$$p_1 := \inf_{y \in \set{ \xi = 0} }\P_y \p{S_1< S_{-1}} > 0 .$$

Obviously, since $Z_s \leq \xi(Y_s)$ and $\xi(y) = 0$,  
$$\P_y(S_1< S_{-1}) \geq \P_y(S^Z_1 < S^Z_{-1})%\geq \P_y \p{S^Z_1 < S^Z_{-1}  } 
=  \P(L^{\tilde Z}_1 < L^{\tilde Z}_{-1} | \tilde Z = 0) > 0, $$
the last term being independent of the choice of $y \in \set{\xi = 0}$.

\section{Remarks on the main result of~\cite{cdgr2}}\label{sec:cdgr2}
Let us now explain the connection between Assumptions~\ref{ass:H1} and~\ref{ass:H2}, and the set of assumptions in~\cite{cdgr2}. Theorem~\ref{th:cdgr2} of the present paper corresponds exactly to Proposition~3.3 and Proposition 3.13 in~\cite{cdgr2} where they are established under Assumption~\ref{ass:H2}, also called Assumption~(B) in~\cite{cdgr2}, as well as a weaker variant of Assumption~\ref{ass:H1}, called Assumption~\ref{ass:J}, and recalled below.

\begin{Def}
First, let us fix a measurable bounded function $\ph: F \to \R$, and denote for each $1 \leq n \leq N$ and any $t\in[0,1]$,
\begin{align*}
&\L^n_t \eqdef Q^{1-t}(\ph)(X_{t}^n) = q(\ph)(X_t^n),%&\L_t \eqdef \frac{1}{N}\sum_{n=1}^N\L^n_t,
\end{align*}
where $\ph$ is omitted in order to lighten the notation. For any $n \in \set{1, \ldots , N}$ and any $k \geq 0$, we denote by $\tau_{n,k}$
the $k$-th branching time of particle $n$, with the convention $\tau_{n,0} =0$. Moreover, for any $j \geq 0$, we denote by $\tau_{j}$ the $j$-th branching time of the whole system of particles, with the convention $\tau_{0} =0$.
\end{Def}

\begin{Rem}
The identity in the definition of $\L^n_t$ comes from Lemma~\ref{lem:dirichlet_op}.
\end{Rem}

%\begin{Rem}
%From the constuction of the particle system, we have that $\xi( (X_{t}^n))=\xi( (X_{t^+}^n))=t$ for all 
%$n \in \set{1, \ldots , N}$ and $t\in  [0,1]$. This explains that $Q^{1-t}(\ph)(X_{t}^n) = q(\ph)(X_t^n)$ 
%according to Lemma~\ref{lem:dirichlet_op}. 
%{\red Je ne suis pas s\^ur que cette rq et le lemma~\ref{lem:dirichlet_op} soient bien utiles} 
%\end{Rem}

A key assumption on the Fleming-Viot particle system in~\cite{cdgr2} is the following.
\renewcommand{\theAssN}{(A')}
\begin{AssN}\label{ass:J}  We assume that  the particle system is such that for the bounded test function $\ph$, $t \mapsto \L^n_t$ is c\`adl\`ag for each $1 \leq n \leq N$, and:
\begin{itemize}
\item[(i)] Only one particle is killed at each branching time:  if $n\neq m$, then $\tau_{n,k}\neq\tau_{m,j}$ almost surely for any $j,k \geq 1$. In other words, the particle system is well-defined.
\item[(ii)] The processes $\L_t^n$ and $\L_t^m$ don't jump at the same time: if $n\neq m$, then
$$\P( \exists t \geq 0, \, \Delta \L^m_t \neq 0\ \&\ \Delta \L^n_t \neq 0)= 0.$$
\item[(iii)] The process $\L_t^n$ never jumps at a branching time of another particle: if $n\neq m$, then
$$ \P( \exists  j \geq 0, \, \Delta \L^n_{\tau_{m,j}} \neq 0)= 0.$$
\end{itemize}
\end{AssN}

In order to obtain precisely Theorem~\ref{th:cdgr2}, it remains to show that Assumption~\ref{ass:H1} implies Assumption~\ref{ass:J}, that is
$$ (\tilde A) \Rightarrow (A').$$

In fact, this can be proven using exactly the same arguments as the ones used to prove Lemma~3.1 in \cite{cdgr2}. In the latter, it is shown that a slightly stronger but very similar assumption (denoted there Assumption~(A), and not specific to the AMS context) implies Assumption~\ref{ass:J},that is
$$ (A) \Rightarrow (A'). $$
However, the proof of $(\tilde A) \Rightarrow (A')$ is very similar to the one of $ (A) \Rightarrow (A')$, so we will not go into more details on this point. In summary, it can be checked, following the arguments of the proof of Lemma~3.1 in \cite{cdgr2}, the chain of implications $$ (A) \Rightarrow (\tilde A) \Rightarrow (A').$$

\section{Assumptions~\ref{ass:feller}, \ref{ass:up}, \ref{ass:minor} imply \ref{ass:H1}  and~\ref{ass:H2}}\label{sec:check_ass}

\begin{Lem}\label{lem:jump_dens}
Under Assumptions~\ref{ass:feller} and~\ref{ass:up}, for any $y \in F=\set{0 \leq \xi \leq 1}$ and any $t \in ]0,1]$ satisfying $\xi(y) \leq t$, one has $\P_y(S_{t} = S_{t^-})=1$, meaning that the jump times of $t \mapsto S_t$ have an atomless distribution.
\end{Lem}

\begin{proof} 
Let us recall that 
$$S_t \eqdef \inf \{ s \geq 0,\  \xi(Y_s) > t \}=\inf \{ s \geq 0,\  \xi(Y_s) = t \}$$
is a stopping time with respect to the natural filtration of $Y$ for all $t\in[0,1]$, and that by continuity of $(Y,\xi)$, the process $t \mapsto S_t$ is c\`adl\`ag. By construction, for $t > 0$,  $S_{t^-}$ is the supremum of the increasing sequence of stopping times $(S_{t-1/k})_{k \geq 1}$, and thus is itself a stopping time. \medskip

% 
% first introduce
% %%$$\widetilde{S}_t \eqdef \inf \p{s \geq 0\mid  \max( \xi(Y_{s^-}),\xi(Y_s)) \geq t }.$$
% $$\widetilde{S}_t \eqdef \inf \{ s \geq 0,\  \xi(Y_s) \geq t \},$$
% which by definition satisfies
% \begin{equation}\label{eq:compS}
% \widetilde{S}_t \leq S_t \eqdef \inf \{ s \geq 0,\  \xi(Y_s) < t \}.
% \end{equation}
% % By continuity of $(Y,\xi)$, $\xi(Y_{\widetilde{S}_t}) = t$ (resp. $\xi(Y_{{S}_t}) = t$) if $\widetilde{S}_t < + \infty$ (resp. $S_t < + \infty$). Let us denote as usual  $S_{t^+}=\lim_{h\to 0^+}S_{t+h}$ and $S_{t^-}=\lim_{h\to 0^+}S_{t-h}$. We define $\widetilde S_{t^+}$ and $\widetilde S_{t^-}$ similarly. 
% By definition, $S$ and $\widetilde S$ are increasing functions of $t$, and by continuity of $Y$ and $\xi$, it can be checked that $S_{t^+} = S_t$ ($S$ is c\`adl\`ag), while $\widetilde S_{t^-}=\widetilde S_t$ ($\widetilde S$ is l\`adc\`ag). On the other hand, by definition we get $S_t \leq \widetilde S_{t^+} $ and $ \widetilde S_t \leq S_{t^+}$, so that keeping in mind~\eqref{eq:compS} it yields $S_t = \widetilde S_{t^+}$ and $$ S_{t^-} = \widetilde S_t,$$
%   %By definition, $S_t=\widetilde S_{t+}$, hence $S_{t}=S_{t+}$. 
%   %On the other hand, by continuity of $Y_t$ and $\xi$, with probability one, 
%   %$\widetilde S_{t-}=S_{t^-}$.
%   the latter being useful in what follows.
  %$$ \widetilde S_{t-}=S_{t^-} \le S_{t-} \le S_t;$$
  %ote also that $\xi(Y_{S_{t^-}})=t$ by continuity of $(Y,\xi)$.
  $Y$ is Feller by Assumption~\ref{ass:feller}, so in particular it is strong Markov: for the stopping time $S_{t^-}$, this gives
  $$ \P_y(S_t = S_{t^-} )= \P_y \p{S_t = S_{t^-} = + \infty} + \E_y \b{\un_{S_{t^-} < + \infty} \P_{Y_{S_{t^-}}}\p{S_t = S_{t^-}}}.$$
But  \eqref{eq:up} in Assumption~\ref{ass:up} directly implies that $\P_{Y_{S_{t^-}}}\p{S_t = S_{t^-}= 0}=1$ almost surely, so that $\P_y(S_t = S_{t^-} )=1$.\end{proof}

\begin{Lem}\label{lem:H1i} 
  Assumptions~\ref{ass:feller} and~\ref{ass:up} imply Assumption~\ref{ass:H1}(i). In other words, for any $h  \geq 0 $ and any initial condition $X_0=x \in F$, one has $X_h = X_{h^-}$ almost surely.
\end{Lem}

\begin{proof}
Consider the level $t_h := 1 \wedge (\xi(x)+h)$. Then $X_{h}= Y_{S_{t_h}}$ where $Y_0=x$, with the convention $Y_{+\infty} = \partial$. Since $h \mapsto t_h$ and $s \mapsto Y_s$ are almost surely continuous, the result is then a consequence of the fact that $S_{t_h} = S_{t_{h}^-}$, which is precisely the result of Lemma~\ref{lem:jump_dens}.
\end{proof}

\begin{Lem}\label{lem:H1ii}
Assumptions~\ref{ass:feller} and~\ref{ass:up} imply Assumption~\ref{ass:H1}(ii), that is to say if $\ph: \set{ \xi = 1} \to \R$ is continuous and bounded, then the integral operator
 $$
 y \mapsto q(\ph)(y) \eqdef \E_y \b{\ph(Y_{S_1}) \un_{S_1 < S_A}}
 $$
 is continuous on the set $\set{0 \leq \xi \leq 1}$.
\end{Lem}

\begin{proof} Consider Lemma~\ref{lem:time_fin}. Letting $B := A \cup \set{ \xi > 1}$, we may write
$$q(\ph)(y)=\E_y \left[\ph(Y_{S_B})\un_{\xi(Y_{S_B})\ge 1} \un_{S_B < \infty}\right].$$ 
The result is now a direct consequence of Lemma~\ref{lem:time_fin}, 
because Assumption~\ref{ass:up} guarantees that  $S_{\{\xi\ge 1\}} = S_{\{\xi>1\}}$ and $S_{\bar A} = S_{\mathring A}$. Since ${\bar B}\subset{\bar A}\cup\{\xi\ge 1\}$, we deduce that $S_{\bar B} = S_{\mathring B}$.
\end{proof}

\begin{Lem}\label{lem:H2} Assumption~\ref{ass:minor} implies Assumption~\ref{ass:H2}, meaning that the number of branchings on the time interval $[0,1]$ is almost surely finite.
\end{Lem}
\begin{proof} 
Define
\begin{align*}
\varepsilon:=\inf_{y: \, 0 \leq \xi(y) \leq 1}\P_{y}\p{ S_{1} < S_A}>0.
\end{align*}
Denote ${\cal J}_n$ the total number of branchings of particle $n$ during the algorithm, and as before $J_1=\sum_{n=1}^N {\cal J}_n$ the total number of branchings. Clearly, we have that $\P({\cal J}_n\geq j)\leq (1-\varepsilon)^j$, so
$$\E[{\cal J}_n]=\sum_{j=1}^\infty \P({\cal J}_n\geq j)\leq 1/\varepsilon.$$ 
We conclude that $\E[J_1]\leq N/\varepsilon<+\infty$, as desired.

%Let $\nu_j$ be the number of particle for which $\tau_{n,j}<1$. 
%The algorithm stops at step $j$ as soon as $\nu_j=0$. Define
%\begin{align*}
%\varepsilon=\inf_{y: \, 0 \leq \xi(y) \leq 1}\P_{y}\p{ S_{1} < S_A}>0.
%\end{align*}
%By the Markov property, the probability of reducing $\nu_j$ is
%at least $\varepsilon$ times the probability of picking a trajectory $M_j$ 
%in the algorithm for which $\tau_{M_j,j}<1$ (cf. Section~\ref{sec.algo}):
%\begin{align*}
%\P(\nu_{j+1}=\nu_j-1|\nu_j)\ge\varepsilon\frac{\nu_j}N.
%\end{align*}
%In particular
%\begin{align*}
%\P(\nu_{j+1}=\nu_j-1|\nu_j>0)\ge\frac{\varepsilon}N.
%\end{align*}
%Hence
%\begin{align*}
%\frac{\varepsilon}N\P(\nu_j>0)\le\P(\nu_{j+1}=\nu_j-1)
%\end{align*}
%and summing over $j$, if $J\le\infty$ stand for the duration of the algorithm,
%\begin{align*}
%&\frac{\varepsilon}N\E[J]\le N\E\Big[\sum_j1_{\nu_{j+1}=\nu_j-1}\Big]\le N.\qedhere
%\end{align*}
\end{proof}

\section{Removing the condition $S_A \wedge S_1 < + \infty$}\label{sec:non_stop_path}

The following property enables us to deal with transient cases where the condition $S_A \wedge S_1 < + \infty$ is not satisfied almost surely, which means that the event $S_t=S_A=+ \infty$ may happen with positive probability.
\begin{Lem}\label{lem:ext}
 Under Assumptions~\ref{ass:feller}, \ref{ass:up} and~\ref{ass:minor}, then almost surely, either $S_1 < + \infty$ or $\xi(Y_s) < 0$ for $s$ large enough, that is $\sup\set{ s \geq 0,\ \xi(Y_{s}) \geq 0} < + \infty$. 
\end{Lem}

\begin{proof}
  Let $y=Y_0 \in \set{0 \leq \xi \leq 1}$ be any initial condition. %Without loss of generality, we can assume that $\xi(y) \leq 1$, otherwise $S_1 = 0$ almost surely (Assumption~\ref{ass:up}). 
 By Assumption~\ref{ass:minor}, we have
  $$2\eps: = \inf_{y \in \set{\xi = 0} }\P_{y}\p{ S_{1} < S_A}=\inf_{y \in \set{0\leq \xi \leq 1} }\P_{y}\p{ S_{1} < S_A} > 0.$$
  In particular, this implies that 
  $$\sup_{y \in \set{0\leq \xi \leq 1}}\P_y(S_1=+\infty)\leq 1-2\eps.$$
  For each $y$, a simple dominated convergence argument shows that we can construct a measurable function $s: \set{0 \leq \xi \leq 1} \to \N_+$ such that 
  $$
  \sup_{y \in \set{0 \leq \xi \leq 1}}\P_{y}\p{ S_{1} < s(y)} \leq 1-\eps.
  $$ 
  Consider the increasing double sequence of stopping times
  $$ \sigma^{1}_0:=0 < \sigma_{1}^1 \leq \sigma^2_0 < \sigma^2_{1} \leq  \ldots $$
  defined for each $n \geq 1$ by
  $$
  \begin{cases}
   \sigma^n_0 = \inf \set{s \geq \sigma^{n-1}_1,\ Y_s \in \set{\xi \geq 0}} \\
   \sigma^n_1 = \sigma^n_0 + s(Y_{\sigma^n_0}) %\sigma(Y_{\sigma^n_0}) \vee \inf \set{s \geq \sigma^{n}_0 | Y_s \in A} 
  \end{cases}
  $$
  By construction, we have the implication
  $$ \sup\set{ s \geq 0,\ \xi(Y_{s}) \geq 0} = + \infty\ \Rightarrow\ \sup_n\sigma^n_0 < + \infty,   $$
  so that it remains to prove that
  %\begin{equation}\label{eq:res_ext}
 $$ \P(\{S_1 = + \infty\}\cap\{\sup_n\sigma^n_0 < + \infty\})= 0.$$
  %\end{equation}
  The strong Markov property as well as the definition of $s(y)$ imply that 
  \begin{align*}
   & \P(\{S_1 \geq \sigma^{n+1}_0\}\cap\{\sigma^{n+1}_0 < + \infty\}  | Y_{\sigma^{n}_0},S_1 \geq \sigma^{n}_0,\sigma^{n}_0 < + \infty)  \\
   &\qquad \leq \P(S_1 \geq \sigma^{n}_1 | Y_{\sigma^{n}_0},S_1 \geq \sigma^{n}_0,\sigma^{n}_0 < + \infty) \leq 1-\eps. 
  \end{align*}
  Iterating the conditioning yields
$$ \P(\{S_1 \geq \sigma^{n+1}_0\}\cap\{\sigma^{n+1}_0 < + \infty\}) \leq (1-\eps)^{n+1}, $$
%| Y_{\sigma^{n}_0}, \quad S_1 \geq \sigma^{n}_0 } \leq 1-\eps $ for each $n \geq 1$. As a consequence, denoting $N_0$ is the first index $n$ for which $\sigma^n_0 = + \infty$,
so that by $\sigma$-additivity 
  $$ \P(\{S_1 \geq \lim_n  \sigma^{n}_0\}\cap\{\sup_n\sigma^{n}_0 < + \infty\})=0.$$
The result follows since $\lim_n  \sigma^{n}_0  = + \infty$.
\end{proof}

\section{A variant of Assumption~\ref{ass:minor}}\label{sec:var_ass}
The following variant of Assumption~\ref{ass:minor} may be useful in practice.
\renewcommand{\theAssN}{3'}
\begin{AssN}\label{ass:minor_bis}  There exists $t_0 \leq 0$ such that $A\subset \{\xi< t_0\}$, the level set $\set{\xi = t_0}$ is compact, and 
\begin{align}\label{eq:minor}  
&\forall y\in \set{ \xi = t_0},~\P_y(S_1 < S_{\bar A}) > 0.
\end{align}
\end{AssN}

Indeed, one has the following implication of assumptions.
\begin{Lem}\label{lem:2point4}
If Assumptions \ref{ass:feller}, \ref{ass:up}, and \ref{ass:minor_bis} are satisfied, then so is Assumption~\ref{ass:minor}.
% Assume that (i) $\set{y \in E \ | \ 0 \leq \xi(y) \leq 1}$ is compact, (ii)  there is a positive number $\delta > 0$ 
% such that 
% $$  \P_{y}\p{ S_{1} < S_A } > 0 \quad  \forall y:  0 \leq \xi(y) \leq 1,$$
% {\blue le $\delta$?}
% 
% (iii) Assumption~\ref{ass:cont} holds true for constants. Then Assumption~\ref{ass:minor} holds true.
\end{Lem}
% We begin with the proof of Lemma \ref{lem:2point4}, which says that 
%  if Assumptions \ref{ass:feller}, \ref{ass:up}, and \ref{ass:minor_bis} hold true, then Assumption~\ref{ass:minor} is satisfied.
% Assume that (i) $\set{y \in E \ | \ 0 \leq \xi(y) \leq 1}$ is compact, (ii)  there is a positive number $\delta > 0$ 
% such that 
% $$  \P_{y}\p{ S_{1} < S_A } > 0 \quad  \forall y:  0 \leq \xi(y) \leq 1,$$
% {\blue le $\delta$?}
% 
% (iii) Assumption~\ref{ass:cont} holds true for constants. Then Assumption~\ref{ass:minor} holds true.
\begin{proof}
First, suppose that Assumptions~\ref{ass:feller} and~\ref{ass:up} are satisfied. We claim that the mapping 
$$
 y \mapsto \P_y(S_1 < S_{\bar A})
$$
is lower semi-continuous on $\set{ 0 \leq \xi \leq t_0}$, in the sense that if $y^n \to y$, then
$$ \P_{y}(S_1 < S_{\bar A})  \leq \liminf_n \P_{y^n}(S_1 < S_{\bar A}).$$ 
Note that Lemma~\ref{lem:H1ii} already implies that this mapping is continuous on $\set{0 \leq \xi \leq 1}$.\medskip

The proof of the claim is similar to the one of Lemma~\ref{lem:time_fin}. Indeed, using Lemma~\ref{lem:cont_feller} and a Skorokhod embedding argument, a sequence $(Y^n_s)_{s \geq 0}$ of Feller processes with initial conditions $(y^n)_{n \geq 0}$ can be constructed on a single probability space so that $\lim_n Y^n = Y$ in $C(\R_+,E)$ almost surely, where $Y$ denotes the Feller process with initial condition $y$. Then Lemma~\ref{lem:stop_time_cont} with~\eqref{eq:time_cond} implies that $\lim_n S^n_{1} = S_1$ as well as $\liminf_n S^n_{\bar A} \geq S_{\bar A}$. But obviously
\begin{align*}
  \set{S_1 < S_{\bar A}}  \subset \set{S_1 < \liminf_n S^n_{\bar A} } \subset \bigcup_{N} \bigcap_{n \geq N} \set{S^n_1 < S^n_{\bar A} },
\end{align*}
so that $\P \p{S_1 < S_{\bar A}} \leq \liminf_n \P\p{S^n_1 < S^n_{\bar A} }$, hence the claim.\medskip

Next, suppose that Assumptions~\ref{ass:feller}, \ref{ass:up}, and~\ref{ass:minor_bis} hold true. For any initial condition $y \in \set{0 \leq \xi \leq 1}$,  denote $\sigma_0 \eqdef \inf \{ s \geq 0,\ \xi(Y_s) =t_0\}$. By the strong Markov property for Feller processes, we may write:
\begin{align*}
 \P_y(S_1 < S_{\bar A}) &= \E_y \b{\un_{\sigma_0 < S_1} \P_{Y_{\sigma_0}}(S_1 < S_{\bar A})}+\P_y\p{S_1 < S_{\bar A}, S_1 \leq \sigma_0} \\
 & \geq \P_y(\sigma_0 < S_1)\inf_{z \in \set{\xi = t_0}} \P_{z}(S_1 < S_{\bar A})+(1-\P_y(\sigma_0 < S_1))\\
 & \geq \inf_{z \in \set{\xi = t_0}} \P_{z}(S_1 < S_{\bar A}).
\end{align*}
%The continuity property of Lemma~\ref{lem:cont_dir} applied with $\ph=1$ and the compacity of $\set{\xi =t_0}$ 
%imply that $\inf_{z \in \set{\xi =t_0}}\P_{z}\p{ S_{1} < S_{\bar A} } > 0$, 
Because a lower semi-continuous function on a compact  reaches its infimum, and using Assumption~\ref{ass:minor_bis}, we get that the latter infimum is $>0$,
hence Assumption~\ref{ass:minor}.
\end{proof}

%The second result follows by continuity of $Y$.
%Finally, the definition of the Skorokhod topology implies that $Y^n_{S^n_O} \in O$ is relatively compact with adherence $\set{ Y_{S_O^-} , Y_{S_O} }$; but~\eqref{eq:time_cond_2} implies that a subsequence cannot converge to $Y_{S_O^-}$ if $Y_{S_O^-} \neq Y_{S_O}$.
%\end{proof}

\section{Proof of Theorem~\ref{th:clt_path}}\label{sec:path_proof}
Theorem~\ref{th:clt_path} is a pathwise extension of Theorem~\ref{gamma} and is stated under the same set of assumptions, namely Assumptions~\ref{ass:feller}, \ref{ass:up}, and~\ref{ass:minor}. The proof follows the same line 
as the latter. The main difference consists in the definition of the level-indexed process and its state space, which is augmented in order to include pathwise information. \medskip

Once the appropriate definition of the level-indexed objects is set up, the interpretation of the AMS algorithm 
as a Fleming-Viot particle system is strictly identical to the specific case described in Section~\ref{sec:FV}. 
From there, it is then sufficient to check Assumptions~\ref{ass:H1} and~\ref{ass:H2} once again thanks to 
Assumptions~\ref{ass:feller}, \ref{ass:up}, and~\ref{ass:minor}, but in a more general pathwise context. \medskip

First, we define the extended level-indexed process as well as its state space. 
To do so, Definition~\ref{def:lev_ind} is generalized as follows. The extended state space $F \cup \set{ \partial}$ 
is now defined by
\begin{align}
F \eqdef& \left\{  y_{[0,s]} \in C_{\rm stop}(\R_+,\set{0 \leq \xi \leq 1} ) \right. \nonumber \\ &\quad \left.  \text{ such that } s < + \infty, \, \xi(y_0) = 0 
\text{\, and \,}\forall s'\in  \left[0,s\right),\, \xi(y_{s'}) <  \xi(y_s)\right\}, \label{eq:F_path}
\end{align}
which is the set of trajectories where the maximum value of $\xi$ is reached only at the endpoint. 
This specific choice of the state space $F$ is adapted to the following construction of the level-indexed process. An initial condition $\calX_0 \eqdef y_{[0,s]} \in F$ being given, we define the level-indexed process as
\[
\calX_h \eqdef 
  \begin{cases}
    Y_{[0,S_{(\xi(y_s) + h)\wedge 1}]}  &\mbox{if}\,  S_{(\xi(y_s) + h)\wedge 1} < S_A,  \\
    \partial &\mbox{if}\,  S_{(\xi(y_s) + h)\wedge 1} \geq S_A.
  \end{cases}
\]
In the above, we have taken as initial condition $Y_{[0,s]}=y_{[0,s]}$ in order to define the underlying Feller process $Y$. Note that in the simpler, usual case, where the initial condition is $\calX_0 = y_0 \in \set{\xi = 0}$, then 
$$
\calX_t = \begin{cases}
    Y_{[0,S_{t}]}  &\mbox{if}\,  S_{t} < S_A,  \\
    \partial &\mbox{if}\,  S_{t} \geq S_A.
  \end{cases}
$$
As before, this rather complicated definition of $\calX$ is required in order to interpret it as a time homogeneous Markov process. 
\begin{Lem}\label{lem:calX}
The set $F$ defined by~\eqref{eq:F_path} is a Borel subset of the Polish space $C_{\rm stop}\p{\R_+,\set{0 \leq \xi \leq 1}}$. $(\calX_h)_{h \geq 0}$ is a c\`adl\`ag process taking values in $F \cup \partial$, which is time homogeneous Markov with respect to its natural filtration.
\end{Lem}
\begin{proof}
First, $F$ can be constructed using the countable intersection of open subsets of the form
$$\set{ y_{[0,s]} \in C_{\rm stop}\p{ \R_+,\set{0 \leq \xi \leq 1}}, \;\; \xi(y_{ s'}) <  \xi(y_s) \;\forall s' \leq 0 \vee (s - 1/q) },$$ where $ q \in \N^\star$. As a consequence $F$ is a Borel subset.\medskip

Second, as in Section~\ref{sec:FV}, the time homogeneous Markov property is a direct consequence of the strong Markov property of $Y$.

\end{proof}

We now wish to check Assumptions~\ref{ass:H1} and~\ref{ass:H2} in order to prove the pathwise CLT Theorem~\ref{th:clt_path}. 

\begin{Lem}
 Under Assumptions~\ref{ass:feller}, \ref{ass:up}, and~\ref{ass:minor}, Assumptions~\ref{ass:H1} and~\ref{ass:H2} hold true for the pathwise level-indexed process $\calX$.
\end{Lem}
\begin{proof}

The fact that Assumption~\ref{ass:H2} follows from Assumption~\ref{ass:minor} has already been established in Lemma~\ref{lem:H2}.\medskip

Assumption~\ref{ass:H1}(i) in the pathwise case is similar to the proof of Lemma~\ref{lem:H1i}, which follows from Lemma~\ref{lem:jump_dens}, namely the fact that $S_{t^-}=S_t$ almost surely. Let us give some details. Let $\calX_0 = y_{[0,s_0]}$ be a given initial condition with initial level $t_0 = \xi(y_{s_0})$. The topology of the space $C_{\rm stop}(\R_+,\set{0 \leq \xi \leq 1} )$ implies that the mapping $s \mapsto y_{[0,s]} \in F $ which spans the same trajectory with different end times is continuous at $s=s_1$ if $ s \mapsto y_s$ is. As a consequence, as in Lemma~\ref{lem:H1i}, since $Y$ is a continuous trajectory, and $t \mapsto S_t$ is a c\`adl\`ag  increasing process, $\calX$ is also c\`adl\`ag and has a jump at $h$ only if $S_{(t_0+h)\wedge 1}$ has one. The proof then follows from Lemma~\ref{lem:jump_dens}. \medskip
%\medskipfrom the fact that the jumps of $\calX$ are contained in the jumps of $t \mapsto {S_t}$ and from the proof of. Moreover, for each level $t_h \in [0,1]$, if $S_{t^-}=S_t$ then $t \maspto Y_{[0,S_{t}]} \in F$ is continuous (no jump) at $t$.  \medskip

The only new technical point is to check Assumption~\ref{ass:H1}(ii), that is to say the continuity of
$$
y_{[0,s]} \mapsto \E \b{ \psi(Y_{[0,S_1]}) \one_{S_1 < S_A} | Y_{[0,s]} = y_{[0,s]}},
$$
where $y_{[0,s]} \in F$, $\psi$ is continuous and bounded on $C_{\rm stop}\p{\R_+,\set{0 \leq \xi \leq 1}}$. This is a pathwise version of Lemma~\ref{lem:H1ii}, and in fact a consequence of the pathwise continuity property stated in Lemma~\ref{lem:time_fin_path} which follows from Assumption~\ref{ass:up}. Indeed the latter states that if $\lim_n y^n_{[0,s^n]} = y_{[0,s]}$ is a converging sequence of initial conditions in $F$, and $\tilde \psi$ is a continuous functional on $C_{\rm stop}(\R_+,\set{0 \leq \xi \leq 1} )$ then $\E \b{ \tilde \psi(Y_{[0,S_1 \wedge S_A ]})| Y_{[0,s]} = y^n_{[0,s^n]}}$ is converging to the corresponding limit. It remains to remark that Assumption~\ref{ass:H1}(ii) is precisely this continuity property for the functional
$$
\begin{cases}
 \tilde \psi(y_{[0,s]}) = \psi(y_{[0,s]}) \qquad &\text{if \, } \xi(y_s) = 1 \\
 \tilde \psi(y_{[0,s]}) = 0 \qquad &\text{if \, } y_s \in \bar A
\end{cases}
$$
which is indeed continuous under the $C_{\rm stop}$ topology.\medskip

Note that we have assumed that $S_1 \wedge S_A<+\infty$ according to Assumption~\ref{ass:minor}. Otherwise, $\tilde \psi$ must be extended by $0$, the continuity of the extension following from Lemma~\ref{lem:ext}.
% . in  a Skorokhod embedding argument enables to construct on the same probability space a sequence of trajectories $Y^n_{[0,S_1^n \wedge S_A^n]}$ such that $Y^n_{[0,s^n]} = y^n_{[0,s^n]}$, and $Y^n_{[0,S_1^n \wedge S^n_A]}$ almost surely converges to $Y_{[0,S_1 \wedge S_A]}$. Then a dominated convergence argument 
% 
\end{proof}

\paragraph{Acknowledgments.} The authors thank Ismael Bailleul for the reference on stochastic waves~\cite{MR0682731}, and Laurent Miclo for a fruitful discussion at the beginning of this work.

\bibliographystyle{plain}
\bibliography{biblio-cdgr}
\end{document}